\documentclass[12pt,reqno]{amsart}

\usepackage{amsmath}
\usepackage{amssymb}
\usepackage{amsthm}
\usepackage{bbm}
\usepackage{mathabx}
\usepackage{url}
\usepackage{float}
\usepackage{graphicx}
\usepackage{verbatim}
\usepackage{bm}
\usepackage{functan}
\usepackage{geometry}
\usepackage{dsfont}
\usepackage{hyperref}
\hypersetup{
    colorlinks,
    citecolor=black,
    filecolor=black,
    linkcolor=black,
    urlcolor=black
}
\allowdisplaybreaks

\makeatletter
\renewcommand{\@seccntformat}[1]{
  \ifcsname prefix@#1\endcsname
    \csname prefix@#1\endcsname
  \else
    \csname the#1\endcsname\quad
  \fi}
\makeatother
\numberwithin{equation}{section}
\newtheorem{theorem}{Theorem}[section]
\newtheorem{corollary}[theorem]{Corollary}
\newtheorem{lemma}[theorem]{Lemma}
\newtheorem{remark}[theorem]{Remark}
\newtheorem{claim}{Claim}

\newtheorem{proposition}[theorem]{Proposition}

\newcommand{\Ut}{e^{it\Delta}}

\newcommand{\fomega}{f^{\omega}}
\newcommand{\R}{\mathbb{R}}
\newcommand{\na}{\nabla}

\newcommand{\dl}{\delta}
\newcommand{\tht}{\theta}
\newcommand{\om}{\omega}
\newcommand{\Om}{\Omega}
\newcommand{\F}{\mathcal{F}}
\newcommand{\Utms}{e^{i(t-s)\Delta}}

\newcommand{\dd}{\partial}
\newcommand{\intr}{\int_{\R^d}}
\newcommand{\re}{\text{Re}}
\newcommand{\im}{\text{Im}}
\newcommand{\fv}{F+v}
\newcommand{\al}{\alpha}
\newcommand{\rd}{\R^d}
\newcommand{\C}{\mathbb{C}}
\newcommand{\WI}{\dot{W}(I)}
\newcommand{\ho}{\dot{H}^1(\R^d)}
\newcommand{\Z}{\mathbb{Z}}
\newcommand{\db}{\dot{B}}
\newcommand{\p}{\frac{d+2}{d-2}}
\newcommand{\pmo}{\frac{4}{d-2}}
\newcommand{\fuw}{u+w+F}
\newcommand{\Uto}{e^{i(t-t_0)\Delta}}
\newcommand{\ujmo}{u(t_{j-1})}
\newcommand{\dnorm}{{L^2_t L^{\frac{2d}{d+2}}_x}}
\newcommand{\pp}{\frac{2d}{d-2}}
\newcommand{\la}{\lambda}
\newcommand*{\medcap}{\mathbin{\scalebox{1.5}{\ensuremath{{\cap}}}}}
\newcommand{\vp}{\varphi}
\newcommand{\N}{\mathbb{N}}
\newcommand{\bkl}{b_{k,l}}
\newcommand{\cmkli}{\hat{c}^{M,i}_{k,l}}
\newcommand{\f}{\frac}

\newcommand{\vn}{v_n}
\newcommand{\fvn}{F_n+v_n}
\newcommand{\fn}{F_n}
\newcommand{\lt}{L^2}
\newcommand{\tjmo}{t_{j-1}}
\newcommand{\fvl}{F_l+v_l}
\newcommand{\vl}{v_l}
\newcommand{\fl}{F_l}
\newcommand{\uo}{u_1}
\newcommand{\ut}{u_2}

\newcommand{\vpi}{\varphi_i}
\newcommand{\vpj}{\varphi_j}
\newcommand{\tvpi}{\tilde{\varphi}_i}
\newcommand{\mo}{M_0}
\newcommand{\sqd}{\sqrt{d}}
\newcommand{\xw}{X_i(\omega)}

\title{Almost Sure Scattering of the Energy-Critical NLS in $d>6$.}
\author{Katie Marsden}
\date{}

\calclayout
\begin{document}
\raggedbottom

\begin{abstract}
    We study the energy-critical nonlinear Schr\"{o}dinger equation with randomised initial data in dimensions $d>6$. We prove that the Cauchy problem is almost surely globally well-posed with scattering for randomised super-critical initial data in $H^s(\R^d)$ whenever $s>\max\{\frac{4d-1}{3(2d-1)},\f{d^2+6d-4}{(2d-1)(d+2)}\}$. The randomisation is based on a decomposition of the data in physical space, frequency space and the angular variable. This extends previously known results in dimension 4 \cite{spitz2021almost}. The main difficulty in the generalisation to high dimensions is the non-smoothness of the nonlinearity. 
\end{abstract}
\maketitle

\tableofcontents

\section{Introduction}
We consider the defocusing energy-critical nonlinear Schr\"{o}dinger equation (NLS) in dimension $d>6$
\begin{equation}\label{NLS}
\begin{cases}
(i\partial_t+\Delta)u=u|u|^{\pmo}\\
u(0)=f\in H^s(\R^d) \text{, } 0<s<1
\end{cases}
\end{equation}

Here ``defocusing" refers to the plus sign in front of the nonlinearity and ``energy-critical" refers to the fact that the conserved energy
\begin{align}		\label{energy}
E(u(t))=\frac{1}{2}\intr |\na u(t)|^2dx+\frac{d-2}{2d}\intr|u(t)|^{\frac{2d}{d-2}} dx
\end{align}
is invariant with respect to the scaling symmetry
$$u(t,x)\mapsto u_\la (t,x):=\la^{\frac{d-2}{2}} u(\la^2t,\la x)$$
Since the energy scales like the $\dot{H}^1$ norm of $u$, we say the equation has scaling regularity 1.

It was shown in \cite{visan2007defocusing} that equation (\ref{NLS}) is globally well-posed with scattering for initial data in the energy space $\dot{H}^1$, however for $s<1$ this is not in general true. Indeed, Christ-Colliander-Tao showed in \cite{christ2003ill} that the solution operator is not continuous in $H^s$ at $u(0)=0$.

The goal of this paper is to investigate the global well-posedness of equation (\ref{NLS}) below the critical regularity $s=1$ by randomising the initial data. We show that for all $s\in(s_d,1)$, where $s_d$ is a constant depending only on the dimension, the equation is almost surely globally well-posed with respect to a particular randomisation for initial data in $H^s(\R^d)$. We moreover establish almost sure scattering in $\dot{H}^s(\R^d)$ both forwards and backwards in time. The randomisation is based on a decomposition of the initial data in physical space, Fourier space and the angular variable as in \cite{spitz2021almost}.

\begin{sloppypar}
The first almost sure scattering result for an energy-critical dispersive equation with super-critical data was proved by Dodson-L\"{u}hrmann-Mendelson in \cite{dodson2020almost} in the context of the 4D nonlinear wave equation with radial data. This was then extended to the 4D energy-critical nonlinear Schr\"{o}dinger equation with radial data by Killip-Murphy-Visan \cite{killip2019almost} with further improvements by Dodson-L\"{u}hrmann-Mendelson in \cite{dodson2019almost}. Spitz then generalised this to the case of non-radial data in \cite{spitz2021almost} by considering a more complex randomisation, and in this paper we further develop this result to cover dimensions $d>6$.
\end{sloppypar}

The main difficulty we encounter in moving to high dimensions is the non-smoothness of the nonlinearity $u|u|^{\pmo}$. To deal with this, we use an adapted version of the work of Tao-Visan \cite{tao2005stability} in Section \ref{scatteringsection} to study the stability of the energy-critical NLS which is needed to prove a conditional scattering result, since in high dimensions the nonlinearity is not twice differentiable and standard stability techniques are insufficient. We also prove local well-posedness in Section \ref{lwpsection} via a regularisation argument, allowing us to work with higher regularity solutions when proving this scattering condition is satisfied.

The many-fold randomisation procedure we consider in this paper was introduced by Spitz in \cite{spitz2021almost}, however each sub-randomisation had previously been used with success. In particular, the randomisation with respect to a unit-scale frequency decomposition, also known as the \textit{Wiener randomisation}, has been extensively applied to nonlinear Schr\"{o}dinger and wave equations, among others, since its simultaneous introduction by L\"{u}hrmann-Mendelson \cite{luhrmann2014random} and B\'{e}nyi-Oh-Pocovnicu \cite{benyi2015probabilistic,benyi2015wiener}, see also \cite{zhang2012random,pocovnicu2017almost}. Randomisation in the angular variable was introduced in \cite{burq2019randomization} in the context of a wave maps type equation, and randomisation in physical space has in particular had applications to the final state problem of the NLS and other dispersive equations, see for example \cite{nakanishi2018randomized,murphy2019random}. The randomisation we use also involves a dyadic frequency decomposition, however unlike its unit-scale counterpart, randomisation with respect to this decomposition alone has not proved useful since it does not entail any improved integrability.

\subsection{Main Result}
We now state our main result. We will define the randomisation of the initial data fully in the next section. Loosely speaking, for any function $f\in H^s(\R^d)$, its randomisation over a probability space $(\Omega,\mathcal{A},\mathbb{P})$ is an $H^s$-valued random variable $$\Omega\ni\omega\mapsto\fomega\in H^s(\R^d)$$
\begin{theorem}		\label{thm1}
Let $d>6$, $s_d:=\max\{\frac{4d-1}{3(2d-1)},\f{d^2+6d-4}{(2d-1)(d+2)}\}<s<1$. Let $f\in H^s(\R^d)$ and $\fomega$ denote the randomisation of $f$ (defined in Section \ref{randomisationsection}). Then there exists $\Sigma {\subset} \Omega$ with $\mathbb{P}(\Sigma)=1$ such that for every $\om \in \Sigma$ there exists a unique global solution 
$$u(t)\in {\Ut}\fomega+C(\R;H^1(\R^d))$$ to the defocusing energy-critical nonlinear Schr\"{o}dinger equation with initial data $\fomega$
\begin{equation} 		\label{rNLS}
\begin{cases}
(i\partial_t+\Delta)u=u|u|^{\pmo}\\
u(0)=\fomega 
\end{cases}
\end{equation}
Moreover, this solution scatters both forwards and backwards in time, i.e. there exist $u_+$, $u_-\in \dot{H}^s(\R^d)$ such that $$\lim_{t\rightarrow \pm\infty} \|u(t)-\Ut u_\pm\|_{\dot{H}^1(\R^d)}=0$$
\end{theorem}
Observe that $s_d=\f{d^2+6d-4}{(2d-1)(d+2)}$ if and only if $d\leq10$.

\begin{remark}\normalfont
By a solution to equation (\ref{NLS}), we mean a solution to the Duhamel formulation of the equation $$u(t)=e^{it\Delta}f-i\int_0^t e^{i(t-s)\Delta}(u|u|^{\pmo})(s)ds$$ in an appropriate function space.
\end{remark}

\begin{remark}\normalfont
In Theorem \ref{thm1} uniqueness holds in the sense that upon writing the solution $u$ in the form 
\begin{equation}        \label{uform}
u(t)={\Ut}\fomega+v(t)
\end{equation}
with $v\in C(\R;H^1(\R^d)){\medcap}{W}(I)$, where the space ${W}(I)$ will be defined shortly (see Section \ref{function_spaces}), the function $v$ is unique.
\end{remark}

\begin{remark}\label{remark1}\normalfont
By writing a solution $u$ of (\ref{rNLS}) in the form (\ref{uform}) we find that $v$ must satisfy the forced equation
\begin{equation}	\label{rFNLS}
\begin{cases}
(i\partial_t+\Delta)v=(F+v)|F+v|^{\pmo}\\
v(0)=v_0
\end{cases}
\end{equation}
with $F=\Ut\fomega$ and $v_0=0$.
Thus, it is sufficient to study the well-posedness of (\ref{rFNLS}) in $H^1(\R^d)$ under some appropriate conditions on $F$.
\end{remark}

\bigskip
Before going into further details we briefly outline the structure of this paper. In Sections \ref{randomisationsection} and \ref{regularisationsection} we will introduce the randomisation procedure for $\fomega$ and the regularisation that we will use for the nonlinearity.

After discussing some preliminaries in Section \ref{notationsection} we will establish (deterministic) local well-posedness of (\ref{rFNLS}) in Section \ref{lwpsection}, under certain conditions on $F=\Ut \fomega$, via a regularisation argument in the space $\dot{W}$ with norm $$\|v\|_{\dot{W}(I)}:=\|\na v\|_{L^\f{2(d+2)}{d-2}_tL^\f{2d(d+2)}{d^2+4}_x(I\times\R^d)}$$
This is the norm used by Tao and Visan in \cite{tao2005stability} to study the perturbation theory of (\ref{NLS}), and it is a convenient norm to work with for the energy-critical equation. 
The argument will also require the forcing, $F$, to lie in $\dot{W}(I)$. Setting $F=\Ut\fomega$ this represents a gain in derivatives which we obtain via a randomisation-improved radially averaged Strichartz estimate as in \cite{spitz2021almost} (see Section \ref{probsection}).

We remark that $\dot{W}$ is not the optimal space to work in to establish local well-posedness of (\ref{rFNLS}). Indeed, the requirement that $\Ut\fomega$ also lies in $\dot{W}$ represents a gain of $$\frac{(d-1)(d-2)}{(2d-1)(d+2)}$$ derivatives on $\Ut\fomega$. However, when used at its endpoint the randomisation-improved radially averaged Strichartz estimate allows us to gain up to $\frac{d-1}{2d-1}$ derivatives and our method can be extended to obtain almost sure local well-posedness for $$1-\frac{d-1}{2d-1}=\frac{d}{2d-1}< s <1$$
We are not able to acheive twice this gain as in \cite{spitz2021almost} due to the non-algebraic nature of the nonlinearity which prevents a more precise analysis of the equation on dyadic scales.

\bigskip
In Section \ref{scatteringsection} we prove a conditional scattering result. The local well-posedness theory of Section \ref{lwpsection} is accompanied by a scattering criterion: if the solution to (\ref{rFNLS}) satisfies $$\|v\|_{\dot{W}(I^*)}<\infty$$ on its maximal interval of existence $I^*$ then the solution is global and scatters as $t\rightarrow \pm \infty$. In this section we show that this condition is satisfied provided the energy of $v$ is uniformly bounded on $I^*$. To this end we develop a perturbation theory in the space $\dot{W}$ to compare solutions of (\ref{rFNLS}) with those of the ``usual" NLS (\ref{rNLS}), since by \cite{visan2007defocusing} we already have a bound on those solutions in $\dot{W}$ in terms of their energy. Since for $d>6$ the nonlinearity is not twice differentiable, we cannot develop the perturbation theory in the standard way and instead adapt the work in \cite{tao2005stability} on the stability of high dimensional energy-critical Schr\"{o}dinger equations. The suitability of perturbation theory for studying probabilistic Cauchy theory of energy-critical NLS was first observed by Pocovnicu in \cite{pocovnicu2017almost}.

In this section we again work in the space $\dot{W}$ and again this is not optimal. Improving the result for this section would require further notation and not improve the final restriction on $s$ in Theorem \ref{thm1}, so we do not present the optimal case.
 
\bigskip
In Section \ref{energybound} we prove the uniform-in-time energy bound mentioned above, placing the forcing term in spaces with low time integrability as in \cite{spitz2021almost}. We prove this bound via the regularised solutions, since the true solution does not have sufficient regularity to perform the necessary computations (in particular, an explicit differentiation of the energy). 

\bigskip
Finally in Section \ref{bounds_section} we show that $F^\om:=\Ut \fomega$ indeed satisfies all the conditions needed to run the arguments above. This relies on the randomisation-improved radially averaged Strichartz estimate and an almost sure estimate on the gradient in $L^1_t$ Lebesgue spaces, both introduced in \cite{spitz2021almost}.

\subsection{Randomisation Procedure}\label{randomisationsection}		\label{randosection}
We now describe how to construct the random variable $\fomega$ appearing in the main theorem.
\subsubsection{Decomposition in Fourier space, physical space, and the angular variable.}
In what follows, let $f\in L^2(\R^d)$.

We first introduce the physical space decomposition. Let $\varphi:\R^d\rightarrow[0,1]$ be a smooth, radially symmetric function with $\vp(x)=1$ for $|x|\leq \sqrt{d}$ and $\varphi(x)=0$ for $|x|\geq 2\sqrt{d}$.\footnote{The factor of $\sqrt{d}$ is just to ensure that every point of $\R^d$ is in the support of some $\vp_i$.} For $i\in\mathbb{Z}^d$ define 
\begin{equation}
\varphi_i(x):=\frac{\varphi(x-i)}{\sum_{k\in \mathbb{Z}^d} \varphi(x-k)} \label{partition}
\end{equation}
so that $\varphi_i$ has support in $\{x: |x-i|\leq 2\sqrt{d}\}$. We then have the unit scale decomposition of $f$ in physical space,
$$f=\sum_{i\in\Z^d} \varphi_i(x)f(x)$$
Note that this representation holds in both the $L^2$ and the pointwise sense.

We next apply an angular decomposition to each component $\vp_i f$. This requires some set-up. Denote by $E_k$ the space of (surface) spherical harmonics of degree $k$, that is to say restrictions to the unit sphere $S^{d-1}$ of homogeneous harmonic polynomials of degree $k$. Denote $N_k=\text{dim}(E_k)<\infty$. The spaces $(E_k)_{k\geq 0}$ are mutually orthogonal and span $L^2(S^{d-1})$ so we may construct an orthonormal Hilbert basis $$B=(b_{k,l})_{k\in \N,l=1,...,N_k}$$ of $L^2(S^{d-1})$, with each $b_{k,l}\in E_k$. By Theorem 6 of \cite{burq2013injections} (see also Theorem 1 of \cite{burq2014probabilistic} and Theorem 1.1 of \cite{burq2019randomization}), there exists a choice of orthonormal basis $(b_{k,l})_{k,l}$, denoted a \textit{good basis}, such that for any $q\in[2,\infty)$, it holds
\begin{equation}\label{eqn1}
    \|b_{k,l}\|_{L^q(S^{d-1})}\leq C_{q,d}
\end{equation}
for some constant $C_{q,d}$ depending only on the indicated parameters and independent of $k,l$.

Recall also the following property of spherical harmonics (see e.g. \cite{stein2016introduction}): write $x\in\R^d$ as $x=r\tht$, for $r\in[0,\infty)$, $\tht\in S^{d-1}$. Then for any surface spherical harmonic $b_{k,l}\in E_k$ and any radial function $f_0$ such that $f(x)=f_0(r)b_{k,l}(\tht)\in L^2(\R^d)$, the Fourier transform of $f$ is given by $\hat{f}(x)=F_0(r)b_{k,l}(\tht)$, where\footnote{We use the convention $\hat{f}(\xi)=\int_{\R^d}e^{-ix\cdot\xi}f(x)dx$. In particular this is different to the convention in \cite{stein2016introduction}.}
\begin{equation}\label{ft}
F_0(r)=(2\pi)^{\frac{d}{2}}i^{-k}r^{-\frac{d-2}{2}}\int_0^\infty f_0(s)J_{\nu(k)}(rs)s^{\frac{d}{2}} ds
\end{equation}
Here $\nu(k):=\f{d+2k-2}{2}$ and $$J_\mu(r):=\frac{(r/2)^\mu}{\Gamma\left(\frac{2\mu+1}{2}\right)\Gamma\left(\frac{1}{2}\right)}\int_{-1}^{1}e^{irs}(1-s^2)^{\frac{2\mu-1}{2}}ds$$
for all $r>0$ and $\mu>-\frac{1}{2}$.

We are now ready to present the decomposition of $\vp_i f$ with respect to the angular variable. We first decompose $\vp_i f$ on dyadic scales in Fourier space: let $\phi:\R^d\rightarrow[0,1]$ be a smooth function supported in $\{|x|\leq 2\}$ and equal to $1$ on $\{|x|\leq 1\}$. For $M\in 2^{\mathbb{Z}}$ define 
\begin{align}
\phi_M(x):=\phi(x/M)-\phi(2x/M)	\label{physicalcutoff}
\end{align}
such that $\phi_M$ is supported in $\{\frac{M}{2}\leq |x|\leq 2M\}$.
For any $h\in L^2(\R^d)$ set $$P_M h:=\mathcal{F}^{-1}(\phi_M\cdot\hat{h})$$

Consider $P_M(\vp_i f)$. It is convenient to rescale $P_M(\vp_i f)$ to unit frequency by setting 
$$g^M_i=(P_M(\vp_i f))(M^{-1}\cdot)$$
Now fix a good basis $(b_{k,l})_{k\in\N,l=1,\ldots,N_k}$ and consider the decomposition of $\hat{g}^M_i$ with respect to this basis, i.e.
\begin{equation*}
    \hat{g}^M_i(\rho \tht)=\sum_{k=0}^\infty \sum_{l=1}^{N_k} \hat{c}^{M,i}_{k,l}(\rho)b_{k,l}(\tht)
\end{equation*}
with each $\hat{c}^{M,i}_{k,l}$ supported in $[\frac{1}{2},2]$ (by orthogonality). (\ref{ft}) then yields
\begin{equation*}
    g^M_i(r\tht)=(2\pi)^{-d}\mathcal{F}(\widehat{g}^M_i)(-r\tht)=\sum_{k=0}^\infty \sum_{l=1}^{N_k} a_k r^{-\f{d-2}{2}}\left(\int_0^\infty \cmkli(s)J_{\nu(k)}(rs)s^{\f{d}{2}}ds\right)b_{k,l}(\tht)
\end{equation*}
for $a_k=(2\pi)^{-\f{d}{2}}i^k$, using that $\bkl(-\tht)=(-1)^k\bkl(\tht)$. It is useful to observe at this point that 
\begin{equation}
\|g^M_i\|_{L^2(\R^d)}^2=\sum_{k=0}^\infty \sum_{l=1}^{N_k}\|\cmkli\|_{L^2(\rho^{d-1}d\rho)}^2\label{eqn_g}
\end{equation} 
Scaling this back to frequency $M$ we obtain
\begin{equation}\label{EQN1}
    P_M(\vp_if)(r\tht)=\sum_{k=0}^\infty \sum_{l=1}^{N_k} a_k (Mr)^{-\f{d-2}{2}}\left(\int_0^\infty \cmkli(s)J_{\nu(k)}(Mrs)s^{\f{d}{2}}ds\right)b_{k,l}(\tht)
\end{equation}

The final step is to include a unit-scale frequency decomposition. To this end we introduce the operators 
\begin{align}		\label{unitscaleop}
P_j f:=\F^{-1}(\psi_j(\xi)\hat{f}(\xi))
\end{align}
where $\psi_j(\xi):=\vp_j(\xi)$ is as in the physical space decomposition. We make this change of notation in order to clarify the distinction between the decompositions in physical and frequency space. Incorporating these projections into (\ref{EQN1}) we obtain 
\begin{equation*}
    P_M(\vp_if)(r\tht)=\sum_{j\in\Z^d}\sum_{k=0}^\infty \sum_{l=1}^{N_k} a_k M^{-\f{d-2}{2}}P_j\left[r^{-\f{d-2}{2}}\left(\int_0^\infty \cmkli(s)J_{\nu(k)}(Mrs)s^{\f{d}{2}}ds\right)b_{k,l}(\tht)\right]
\end{equation*}
from which
\begin{equation*}
    f(r\tht)=\sum_{M\in 2^{\Z}}\sum_{i,j\in\Z^d}\sum_{k=0}^\infty \sum_{l=1}^{N_k} a_k M^{-\f{d-2}{2}}P_j\left[r^{-\f{d-2}{2}}\left(\int_0^\infty \cmkli(s)J_{\nu(k)}(Mrs)s^{\f{d}{2}}ds\right)b_{k,l}(\tht)\right]
\end{equation*}
with convergence in $L^2(\R^d)$.

\subsubsection{Randomisation with respect to the decomposition.}
\begin{sloppypar}
We now introduce a family $$(X^M_{i,j,k,l}:M\in 2^\Z,i,j\in \Z^d, k\in \N_0,l\in\{1,\ldots,N_k\})$$ of independent, mean-zero, real-valued random variables on a probability space $(\Omega, \mathcal{A}, \mathbb{P})$ with respective distributions $(\mu^M_{i,j,k,l}:M\in 2^\Z,i,j\in \Z^d, k\in \N_0,l\in\{1,\ldots,N_k\})$ for which there exists a $c>0$ such that 
$$\int_\R e^{\gamma x} d\mu^M_{i,j,k,l}(x)\leq e^{c\gamma^2}$$      
for all $\gamma\in \R$, $M\in 2^{\mathbb{Z}}$, $i,j\in \Z^d$, $k\in \N_0$, $l\in\{1,\ldots,N_k\}$. This is satisfied by independent identically distributed Gaussians for example. We can then define the randomisation
\end{sloppypar}
\vspace{-0.6em}
\begin{multline}\label{rand_f}
\fomega=\sum_{M\in 2^\Z}\sum_{i,j\in \Z^d}\sum_{k=0}^\infty\sum_{l=1}^{N_k} X^M_{i,j,k,l}(\om) P_j f^{M,i}_{k,l}\\
\equiv\sum_{M,i,j,k,l} X^M_{i,j,k,l}(\om)a_k M^{-\f{d-2}{2}}P_j\left[r^{-\f{d-2}{2}}\left(\int_0^\infty \cmkli(s)J_{\nu(k)}(Mrs)s^{\f{d}{2}}ds\right)b_{k,l}(\tht)\right]
\end{multline}
which is well-defined in $L^2(\Omega,L^2(\R^d))$.

\begin{remark}\label{rmk1.4} \normalfont
In fact, for $f\in H^s(\R^d)$, the randomisation $\fomega$ also lies in $H^s(\R^d)$ almost surely. In particular, it holds $$\| \|\fomega\|_{H^s(\R^d)}\|_{L^2(\Omega)}\lesssim_d \|f\|_{H^s(\R^d)}$$
This can be seen using the fundamental large deviation estimate of Burq and Tvetkov (see Section \ref{bounds_section}), combined with the orthogonality of the decompositions in frequency space and into spherical harmonics, and Corollary 3.3 of \cite{spitz2021almost} to handle the intertwining of the physical space decomposition and the $H^s$ norm.
In what follows, we implicitly restrict to a subset $\Sigma \subset \Omega$ of probability one such that $\fomega \in H^s(\R^d)$ for every $\omega\in \Sigma$.
\end{remark}

\begin{remark}\label{rmk1.5} \normalfont
It is important to note that the above randomisation does not in general improve the regularity of the data. In particular, choose the probability space $(\Om,\mathcal{A},\mathbb{P})$ to be the product of spaces $(\Om_i,\mathcal{A}_i,\mathbb{P}_i)_{i=1,2,3}$ and the random variables to be given by
\begin{align*}
X^{M}_{i,j,k,l}(\om)=X_j(\om_1)X^M_{k,l}(\om_2)X_i(\om_3), && \om=(\om_1,\om_2,\om_3)\in\Om_1\times\Om_2\times\Om_3
\end{align*}
with the $X_j$, $X^M_{k,l}$, $X_i$ independent identically distributed Bernoulli random variables on $\Om_1$, $\Om_2$, $\Om_3$ respectively taking values $\pm1$ with equal probability $\frac{1}{2}$. Then one can show that, for $0<s<1$, $f\notin H^s(\rd)$ implies that $\fomega\notin H^s(\rd)$ for almost every $\om\in\Om$. See Appendix \ref{appendixB} for further details.

\end{remark}

\subsection{Regularisation of the Nonlinearity}\label{regularisationsection}
We shall study solutions to (\ref{NLS}) via a regularisation of the nonlinearity $g(u):=u|u|^{\f{4}{d-2}}$, allowing us to work with $H^2$ solutions when performing calculations involving the energy later on. This step is not necessary in the lower $4$ dimensional settings of \cite{dodson2019almost} and \cite{spitz2021almost} when the nonlinearity in (\ref{NLS}) is algebraic and persistence of regularity allows us to directly construct a solution in $H^1$ as a limit of solutions in $H^2$.

Denote $p=\f{d+2}{d-2}$, so $p-1=\f{4}{d-2}$. For each $n\in\N$ define $$g_n(u):=u\vp_n'(|u|^2)$$ for $\vp_n(x)=n^{p+1}\vp_1(x/n^2)$. Here $\vp_1\in C^\infty((0,\infty))\cap C([0,\infty))$ with $\vp_1(0)=0$ and
\begin{equation}
\vp_1'(x)=
\begin{cases}
x^{\f{p-1}{2}} \text{ for } 0< x\leq 1\\
2^{p-1} \text{ for } x\geq 4
\end{cases}
\end{equation}
in such a way that $\vp_1'(x)\leq x^\f{p-1}{2}$ for all $x\geq 0$.
Thus
\begin{equation}
\vp_n'(x)=
\begin{cases}
x^{\f{p-1}{2}} \text{ for } 0< x\leq n^2\\
(2n)^{p-1} \text{ for } x\geq (2n)^2
\end{cases}
\end{equation}
and $g_n (u)=g(u)$ whenever $|u|\leq n$. Since $\varphi_1''$ is compactly supported, we also see that $|\varphi_n''(x)|\lesssim |x|^{\f{p-3}{2}}$.

Consider the regularised NLS
\begin{equation}\label{reg_nls}
\begin{cases}
(i\dd_t+\Delta)u_n=g_n(u_n)\\
u_n(0)=u_{n,0}\in H^2(\R^d)
\end{cases}
\end{equation}
By Theorem 4.8.1 of \cite{cazenave2003semilinear} we see that (\ref{reg_nls}) admits a local solution in $C(I,H^2(\R^d))\cap C^1(I,L^2(\R^d))$ on some neighbourhood $I$ of $0$. Here
\begin{equation}\label{C1_space}
C^1(I,L^2(\R^d)):=\{f\in C(I,L^2(\R^d)):\dd_tf\in C(I,L^2(\R^d)) \}
\end{equation}
where $\dd_t f$ is defined as the vector-valued distribution such that $$\int_I \dd_t\psi(t) f(t,\cdot) dt=-\int_I \psi(t) \dd_tf(t,\cdot) dt$$ for all $\psi\in \mathcal{D}(I)$, with the above integrals evaluated in the Bochner sense.

Theorem 5.3.1 of \cite{cazenave2003semilinear} then shows that this solution exists in $H^2$ for as long as it exists in $H^1$, which is for all time since solutions of (\ref{reg_nls}) have conserved energy
$$
E_n(u_n)=\frac{1}{2}\int_{\rd}|\na u_n|^2dx+\frac{1}{2}\int_{\rd}\varphi_n(|u_n|^2)dx
$$
We thus see that (\ref{reg_nls}) admits global solutions in $C(\R,H^2(\R^d))\cap C^1(\R,L^2(\R^d))$.

\bigskip

As discussed in Remark \ref{remark1}, in this paper we will actually study the well-posedness of the forced equation (\ref{rFNLS}) in $H^1(\R^d)$, with the forcing term given by the free evolution of the randomised data: $F=\Ut\fomega$. Thus to obtain $H^2$ solutions to (\ref{rFNLS}), we must also regularise the forcing. 

\begin{sloppypar}
Set $F_n=P_{\leq n} F=\Ut P_{\leq n} \fomega$, where $$P_{\leq n}\fomega:=\sum_{\substack{M\in 2^{\mathbb{Z}}\\M\leq 2^n}}P_M \fomega$$ Then by Lemma 4.8.2 of \cite{cazenave2003semilinear}, $F_n\in C(\R,H^2(\R^d))\cap C^1(\R,L^2(\R^d))$. Observe that for any $1\leq a,b\leq\infty$ it holds $$\|F_n\|_{L^a_tL^b_x(\R)}\lesssim_{a,b,d}\|F\|_{L^a_tL^b_x(\R)}$$
\end{sloppypar}
Fix $v_0\in H^1(\R^d)$. Setting $u_n:=v_n+F_n$ and $u_{n,0}=P_{\leq n}(v_0+\fomega)$ in (\ref{reg_nls}), we thus obtain unique global solutions to the forced NLS
\begin{equation*}
\begin{cases}
(i\dd_t+\Delta)v_n=g_n(F_n+v_n)\\
v_n(0)=v_{n,0}\in H^2(\R^d)
\end{cases}
\end{equation*}
in $C(\R,H^2(\R^d))\cap C^1(\R,L^2(\R^d))$. Here $v_{n,0}:=P_{\leq n}v_0\rightarrow v_0$ in $H^1(\R^d)$.

We will show in Section \ref{lwpsection} that the solutions $\vn$ converge locally to solutions of the non-regularised equation (\ref{rFNLS}).

\section{Notation and Preliminaries}\label{notationsection}
\subsection{Notation}
$C_{\alpha_1,...,\alpha_n}$ denotes a constant depending only on the parameters $\al_1,...,\al_n$ whose precise value may change line to line.
We write $X\lesssim_{\alpha_1,...,\alpha_n}Y$ to mean $X\leq C_{\alpha_1,...,\alpha_n}Y$.

Unless stated otherwise, for $p\in[1,\infty]$, $p'\in[1,\infty]$ denotes its conjugate exponent such that $$\f{1}{p}+\f{1}{p'}=1$$

We use the spacetime Lebesgue spaces $L^p_t L^q_x$ with norms $$\|f\|_{L^p_tL^q_x(I\times\R^d)}\equiv\|f\|_{L^p_tL^q_x(I)}\equiv\|f\|_{p,q[I]}:=\| \|f\|_{L^q_x(\R^d)} \|_{L^p_t(I)}$$
as well as the homogeneous $\ell^2$ Besov spaces with
\begin{equation*}
    \|f\|_{\dot{B}^r_{q,2}(I)}:=\left(\sum_{N\in 2^{\mathbb{Z}}} N^{2r}\|P_Nf\|_{L^q_x(I\times \R^d)}^2\right)^{\frac{1}{2}}  
\end{equation*}
and the mixed spacetime Besov spaces with
\begin{equation*}
    \|f\|_{\dot{B}^r_{p,q,2}(I)}:=\left(\sum_{N\in 2^{\mathbb{Z}}} N^{2r}\|P_Nf\|_{L^p_tL^q_x(I\times \R^d)}^2\right)^{\frac{1}{2}}  
\end{equation*}
Since we shall always be considering the $\ell^2$ Besov-type spaces we will sometimes omit the subscript ``$2$", writing only $\db^r_{p,q}(I)$.

Throughout this paper it is always assumed that $d>6$, and we will often use the notation $p=\f{d+2}{d-2}\in(1,2)$ without comment.

\subsection{Properties of the Nonlinearity}
Denote $g(u):=u|u|^{\pmo}$. We record here some properties of $g$ for future reference. As well as the trivial bound $|g(u)|\leq |u|^{\p}$, we have the gradient bounds $|g_z(u)|\lesssim_d |u|^{\pmo}$, $|g_{\bar{z}}(u)|\lesssim_d |u|^{\pmo}$. Here $g_z,g_{\bar{z}}$ denote the complex derivatives:
\begin{align*}
g_z(x+iy)\equiv \dd_z g(x+iy)=\frac{1}{2}(\f{\dd g}{\dd x}-i\f{\dd g}{\dd y}), && g_{\bar{z}}(x+iy)=\frac{1}{2}(\f{\dd g}{\dd x}+i\f{\dd g}{\dd y})
\end{align*}
for $z=x+iy$, $x,y\in\R$.
We also have the difference bound
\begin{equation}\label{eqn123}
    |g(u_1)-g(u_2)|\lesssim_d (|u_1|^{\pmo}+|u_2|^{\pmo})|u_1-u_2|
\end{equation}
which follows from the identity
\begin{equation}
g(u_1+u_2)-g(u_1)=\int_0^1 [g_z(u_1+\theta u_2) u_2 +g_{\bar{z}}(u_1+\theta u_2)\bar{u}_2] d\theta		\label{gdiff}
\end{equation}
On the other hand, the chain rule
\begin{equation}    \label{chain_rule}
    \na g(u(x))=g_z(u(x))\na u(x)+g_{\bar{z}}(u(x))\na\bar{u}(x)
\end{equation}
with the bound\footnote{Note this bound holds since $d>6$. For $d\leq6$ we have Lipschitz continuity of $g_z$, $g_{\bar{z}}$, making some aspects of the problem simpler to study.}
\begin{equation}\label{eqn556}
    |g_z(u_1)-g_z(u_2)|\lesssim_d |u_1-u_2|^{\pmo}
\end{equation}
(and the analogous statement for $g_{\bar{z}}$), implies that
\begin{equation} 	\label{gderiv}
|\na g(u_1)-\na g(u_2)|\lesssim_d |u_1-u_2|^{\pmo}|\na u_1|+|u_2|^{\pmo}|\na u_1-\na u_2|
\end{equation}

Moreover, the above bounds all also hold for $g_n $ with bounds independent of $n$.

\subsection{Deterministic Estimates}
We first recall the Littlewood-Paley inequality. 
\begin{lemma}(Littlewood-Paley Inequality)		\label{LPIQ}
For any $1<p<\infty$ we have
\begin{equation*}
\|f\|_{L^p(\R^d)}\sim_{p,d} \left\| \left(\sum_{N\in 2^{\mathbb{Z}}} |P_Nf|^2\right)^\frac{1}{2}\right\|_{L^p(\R^d)}
\end{equation*}
\end{lemma}
This lemma allows us to easily transfer between the Besov and standard Lebesgue spaces. In particular, combined with the triangle inequality it yields
\begin{equation}\label{tiq}
\|f\|_{L^p_t L^q_x}\lesssim_{q,d} \left(\sum_{N\in 2^{\Z}}\|P_N f\|_{L^p_t L^q_x}^2\right)^\frac{1}{2}
\end{equation}
for any $2\leq p\leq \infty$, $2\leq q<\infty$. We also have the dual estimate
\begin{equation}\label{rtiq}
\left(\sum_{N\in 2^{\Z}}\|P_N f\|_{L^p_t L^q_x}^2\right)^\frac{1}{2}\lesssim_{q,d} \|f\|_{L^p_t L^q_x}
\end{equation}
for $1\leq p\leq 2$, $1<q\leq2$.

In addition, we will repeatedly use the property $$\||\na|^s P_N f\|_{L^p(\R^d)}\sim_{p,d} N^s \|P_Nf\|_{L^p(\R^d)}$$ for all $1\leq p\leq\infty$, and in the case $s=1$, the Riesz estimate
$$
\||\na|P_Nf\|_{L^p(\rd)}\sim_{p,d}\|\na P_N f\|_{L^p(\rd)}
$$
for $1<p<\infty$.

Recall also the Strichartz estimates:
\begin{lemma}		\label{strichartz}
Let $d\geq 3$. We call $2\leq q,r\leq \infty$ a Strichartz-admissible pair if
\begin{align}	\label{pair}
\frac{2}{q}+\frac{d}{r}=\frac{d}{2}
\end{align}
Let $(q,r)$ and $(\tilde{q},\tilde{r})$ be Strichartz-admissible pairs. 
Denote by $\tilde{q}'$ and $\tilde{r}'$ the conjugate exponents of $\tilde{q}$ and $\tilde{r}$. Let $I{\subset} \R$ be a time interval containing $t_0$. It holds 
\begin{gather}
\|\Ut f\|_{L^q_t L^r_x(\R\times \R^d)} \lesssim_{d,r,q} \|f\|_{L^2(\R^d)}     \label{hom}\\
\left \|\int_I e^{-is\Delta} F(s) ds \right\|_{L^2(\R^d)} \lesssim_{d,\tilde{r},\tilde{q}} \|F\|_{L^{\tilde{q}'}_t L^{\tilde{r}'}_x (I \times \R^d)}	\label{dual}\\
\left \|\int_{t_0}^t \Utms F(s) ds \right\|_{L^q_t L^r_x (I \times \R^d)} \lesssim_{d,r,q,\tilde{r},\tilde{q}} \|F\|_{L^{\tilde{q}'}_t L^{\tilde{r}'}_x (I \times \R^d)} \label{inhom}
\end{gather}
\end{lemma}

\section{Function Spaces}\label{function_spaces}
We now define the function spaces in which we shall place the solution and the forcing in order to obtain local well-posedness.

Let $I$ be an open time interval.
We will place the solution $v$ to the forced NLS into the space defined by the norm
\begin{equation*}
\|v\|_{W(I)}:=\|v\|_{V(I)}+\|\na v\|_{V(I)}
\end{equation*}
where
$$\|v\|_{V(I)}:=\|v\|_{\f{2(d+2)}{d-2},\f{2d(d+2)}{d^2+4}[I]}$$
We will also denote $\|v\|_{\WI}:=\|\na v\|_{V(I)}$.

To prove local well-posedness it will be sufficient to place the forcing term $F$ into the same space $W$. However to obtain the conditional scattering result in Section \ref{scatteringsection} we will need $F$ to lie in the stronger space\footnote{Here we use the classical definition $\|\cdot\|_{X\cap Y}\equiv \|\cdot\|_X+\|\cdot\|_Y$.}
\begin{equation}\label{Rdef}
R(I):=W(I)\cap \dot{B}^{\frac{4}{d+2}}_{d+2,\frac{2(d+2)}{d}}(I)
\end{equation}
which is necessary in order to apply the theory developed in \cite{tao2005stability} to study the stability of the forced equation.

Again we will also denote $\dot{R}(I):=\dot{W}(I)\cap \dot{B}^{\frac{4}{d+2}}_{d+2,\frac{2(d+2)}{d}}(I)$.

Observe that the above norms are continuous in their endpoints and are ``time-divisible" in the sense that for each of the spaces $S(I)$ just introduced there exists a finite constant $\al(S)>0$ such that 
\begin{equation*}
\left( \sum_{j=1}^J \|v\|_{S(I_j)}^{\alpha(S)}\right)^{\frac{1}{\al(S)}} \leq\|v\|_{S(I)}
\end{equation*}
whenever $I$ is the disjoint union of consecutive intervals $(I_j)_{j=1}^J$. In particular, $\al(W)=\al(\dot{W})=\al(V)=\f{2(d+2)}{d-2}$ and $\al(R)=\al(\dot{R})=d+2$ (see, for example, \cite{dodson2019almost}).

We deduce that whenever $\|v\|_{S(I)}<\infty$ for $S$ any of $W,\dot{W},V,R$ or $\dot{R}$, we may partition $I$ into $J$ consecutive intervals $(I_j)_{j=1}^J$ with disjoint interiors such that $$\|v\|_{S(I_j)}\leq \epsilon$$ for each $j=1,...,J$ and $$J\leq 2 \left( \frac{\|v\|_{S(I)}}{\epsilon}\right)^{\alpha(S)}$$

We end this section with the observation that for $F_n=P_{\leq n} F$ the regularised forcing term as in Section \ref{regularisationsection} it holds $\|F_n\|_{S(I)}\lesssim_d\|F\|_{S(I)}$ and $$\|F_n-F\|_{S(I)}\to 0$$ as $n\to\infty$ for $S$ any of the function spaces $W,\dot{W},R,\dot{R}$ or $V$. 

\section{Local Well-Posedness}		\label{lwpsection}
In this section we will prove the deterministic local well-posedness of the problem
\begin{equation}\label{fNLS}
\begin{cases}
(i\partial_t+\Delta)v=(F+v)|\fv|^{\pmo}\\
v(t_0)=v_0\in H^1(\R^d)
\end{cases}
\end{equation}
in $H^1$ under appropriate conditions on the forcing term $F$. 
We will construct solutions via the regularised equation
\begin{equation}\label{FNLSn}
\begin{cases}
(i\dd_t+\Delta)v_n=g_n(F_n+v_n)\\
v_n(t_0)=v_{n,0}\in H^2(\R^d)
\end{cases}
\end{equation}
for $g_n$ as in Section \ref{regularisationsection}.
\subsection{Linear and Nonlinear Estimates}
Let $I\subset\R$ be an interval containing $t_0$.
First observe the following inhomogeneous estimate, which is a direct application of the Strichartz inequality (\ref{strichartz}):
$$\|e^{i(t-t_0)\Delta}v_0\|_{\dot{W}(\R)}\lesssim_d \|v_0\|_{\dot{H}^1}$$

Then by the inhomogeneous Strichartz estimate (\ref{inhom}) followed by the chain rule (\ref{chain_rule}) we have 
\begin{equation*}
\left\|\na\int_{t_0}^t \Utms g_n (u)(s) ds \right\|_{q,r[I]}\lesssim_{d,q,r} \||u|^{{\pmo}}\na u\|_{2,\f{2d}{d+2}[I]}
\end{equation*}
for any Strichartz pair $(q,r)$. Moreover, using H\"{o}lder's inequality and the Sobolev embedding $\dot{W}^{1,\frac{2d(d+2)}{d^2+4}}(\R^d)\hookrightarrow L^{\frac{2(d+2)}{d-2}}(\R^d)$ we observe that
\begin{align}
\||\uo|^{\pmo}\ut\|_{2,\f{2d}{d+2}[I]}\lesssim_d \|\uo\|_{\WI}^{\pmo}\|\ut\|_{V(I)}\label{bound707}
\end{align}
which in particular gives
$$\left\|\int_{t_0}^t \Utms g_n (u)(s) ds \right\|_{\dot{W}(I)}\lesssim\||u|^{\pmo}\na u\|_{2,\f{2d}{d+2}[I]}\lesssim_d\|u\|_{\WI}^{p}$$ for all $u\in \WI$, recalling $p=\f{d+2}{d-2}$.

\subsection{Proof of Local Well-Posedness and Scattering Condition}
\begin{theorem}\label{wellposedness}
There exists $\epsilon_0(d)>0$ such that the following holds. Let $v_0\in H^1(\R^d)$ and $F\in W\cap L^{\infty}_tL^{\f{2d}{d-2}}_x(\R)$. Let $I\ni t_0$ be a sufficiently small time interval such that 
\begin{equation}
\|e^{i(t-t_0)\Delta}v_0\|_{\dot{W}(I)}+\|F\|_{\dot{W}(I)}\leq \epsilon \label{F}
\end{equation}
 for some $0<\epsilon<\epsilon_0(d)$. Then there exists a unique solution $v\in C(I,{H}^1(\R^d))\cap W(I)$ to (\ref{fNLS}) which satisfies $$\| v\|_{\WI}\leq 4 \epsilon$$
This solution extends to a maximal interval of existence $I^*:=(T_-,T_+)$ in this space. Moreover, 
\begin{enumerate}
\item if $T_+<\infty$, then $\|v\|_{\dot{W}([t_0,T_+))}=\infty$
\item if $T_+=\infty$ and $\|v\|_{\dot{W}([t_0,T_+))}<\infty$, then the solution $v$ scatters forwards in time, i.e. there exists $v_+\in \dot{H}^1(\R^d)$ with $$\lim_{t\rightarrow \infty} \|v(t)-\Ut v_+\|_{\dot{H}^1}=0$$
\end{enumerate}
The analogous statements hold for $T_-$. 

On compact subintervals $\tilde{I}$ of $I^*$, $v$ is obtained as a limit in $L^q_tL^r_x(\tilde{I})$ of solutions $\vn$ to the regularised equation (\ref{FNLSn}), for any Strichartz pair $(q,r)$.
\end{theorem}

\begin{proof}
Denote by $\vn$ the unique solution in $C(I,H^2(\R^d))\cap C^1(I,L^2(\R^d))$ to (\ref{FNLSn}) with initial data $v_{n,0}=P_{\leq n}v_0$. 
We will show that $(\vn)_n$ is Cauchy in $V(I)$. Observe that for any  Strichartz pair $(q,r)$ and any $l\geq n$, we have
\begin{align*}
\|\vn-v_l\|_{q,r[I]}
\lesssim&_{q,r,d} \|v_{n,0}-v_{l,0}\|_{L^2(\R^d)}+\|g_n(\fvn)-g_n(F_l+v_l)\|_{2,\f{2d}{d+2}[I]}\\
&+\|g_n (\fvl)-g(\fvl)\|_{2,\f{2d}{d+2}[I]}+\|g(\fvl)-g_l(\fvl)\|_{2,\f{2d}{d+2}[I]}
\end{align*}
We bound each of these terms separately. Firstly, by (\ref{eqn123}) applied to $g_n $ and the nonlinear estimate (\ref{bound707}) we have, for any $u_1$, $u_2\in W(I)$,
\begin{align*}
\|g_n(u_1)-g_n(u_2)\|_{2,\f{2d}{d+2}[I]}\lesssim&_{d} (\|\uo\|_{\dot{W}(I)}^{p-1}+\|\ut\|_{\dot{W}(I)}^{p-1})\|\uo-\ut\|_{V(I)}
\end{align*}
and the analogous bound for $g$.
Next, since $g_n (u)=g(u)$ for $|u|\leq n$, we may bound
\begin{align*}
\|g_n(u)-g(u)\|_{2,\f{2d}{d+2}[I]}\lesssim&_{d} \||{u}|^p\mathds{1}_{|{u}|\geq n}\|_{2,\f{2d}{d+2}[I]}\\
\lesssim&_{d} \|{u}\|_{\dot{W}(I)}^{p-1}\|{u}\cdot\mathds{1}_{|{u}|\geq n}\|_{\f{2(d+2)}{d-2},\f{2d(d+2)}{d^2+4}[I]}\\
\lesssim&_{d} \| {u}\|_{\dot{W}(I)}^{p-1}\|\mathds{1}_{|{u}|\geq n}\|_{\infty,\f{d(d+2)}{4}[I]} \|u\|_{\f{2(d+2)}{d-2},\f{2d}{d-2}[I]}\\
\lesssim&_{d} \|{u}\|_{\dot{W}(I)}^{p-1}
\left(\sup_{t\in I}\f{1}{n^{\f{2d}{d-2}}}\intr|{u}|^{\f{2d}{d-2}}dx\right)^{\f{4}{d(d+2)}}
 |I|^{\f{d-2}{2(d+2)}}\|u\|_{\infty,\f{2d}{d-2}[I]}\\
\lesssim&_{d} |I|^{\f{d-2}{2(d+2)}}n^{-\f{8}{d^2-4}}\|{u}\|_{\dot{W}(I)}^{p-1}
 \|u\|_{\infty,p+1[I]}^{\f{d^2+4}{d^2-4}}
\end{align*}
Thus since $l\geq n$, using that $F_n=P_{\leq n} F$ and $\dot{H}^1(\R^d)\hookrightarrow L^{p+1}(\R^d)$, we have
\begin{align}
\|\vn-v_l\|_{q,r[I]}
\lesssim&_{q,r,d} \|v_{n,0}-v_{l,0}\|_{L^2(\R^d)}\nonumber\\
&+(\|\vn\|_{\dot{W}(I)}^{p-1}+\|\vl\|_{\dot{W}(I)}^{p-1}+\| F\|_{\dot{W}(I)}^{p-1}+\| F\|_{\dot{W}(I)}^{p-1})\|\vn-\vl+\fn-\fl\|_{V(I)}\nonumber\\
&+|I|^{\f{d-2}{2(d+2)}}n^{-\f{8}{d^2-4}}(\|\vl\|_{\dot{W}(I)}^{p-1}+\| F\|_{\dot{W}(I)}^{p-1})
 (\|\vl\|_{L^\infty_t \dot{H}^1(I)}^{\f{d^2+4}{d^2-4}}+\|F\|_{\infty,p+1[\R]}^{\f{d^2+4}{d^2-4}}) \label{cauchy1}
\end{align}

To proceed, we need a bound on $\|\vn\|_{\dot{W}(I)}$. By the nonlinear estimates we have, for any $t_0\in I'\subset I$,
\begin{small}
\begin{align*}
\| \vn\|_{\dot{W}(I')}\leq& \|e^{i(t-t_0)\Delta}  (v_{n,0}-v_0)\|_{\dot{W}(I)}+\|e^{i(t-t_0)\Delta}  v_0\|_{\dot{W}(I)}+C_d(\| \vn\|_{\dot{W}(I')}^p+\| \fn\|_{\dot{W}(I')}^p)\\
\leq& 2\epsilon+C_d\| \vn\|_{\dot{W}(I')}^p
\end{align*}
\end{small}
for $n$ sufficiently large, $\epsilon(d)$ sufficiently small. Taking $\epsilon(d)$ smaller still, 
a standard continuity argument shows that $\|\vn\|_{\WI}\leq 4\epsilon$.

Lastly we observe that
\begin{equation}
\|\vn\|_{L^\infty_t \dot{H}^1_x(I)}\lesssim_d \|v_{n,0}\|_{\dot{H}^1}+\| \vn\|_{\dot{W}(I)}^p+\|F_n\|_{\dot{W}(I)}^p\lesssim_d \|v_0\|_{\dot{H}^1}+\epsilon^p\label{eqn555}
\end{equation}
so the $\vn$ are uniformly bounded in $\dot{H}^1$ on $I$, say by $C(v_0,\epsilon,d)$.

Putting the above estimates into (\ref{cauchy1}) along with the assumption (\ref{F}), we see that $(\vn)_n$ is Cauchy in $L^q_t L^r_x$ for any Strichartz pair $(q,r)$. In particular $(\vn)_n$ has a limit $v\in V(I)$, which still satisfies $\|v\|_{\WI}\leq 4\epsilon$ and solves equation (\ref{fNLS}).

By standard arguments, one may extend $v$ to a maximal interval of existence $(T_-,T_+)$, such that it is the unique solution to (\ref{fNLS}) in $C([\alpha,\beta];H^1_x(\R^d))\medcap W([\al,\beta])$ for any $T_-<\alpha<t_0<\beta<T_+$.

We next prove the blow up criterion. We work forwards in time since the result in the negative time direction is proved in the same way. Suppose that $T_+<\infty$ and $\| v\|_{\dot{W}([t_0,T_+))}<\infty$. Consider a sequence $t_n\nearrow T_+$. Note that
\begin{align*}
e^{i(t-t_n)\Delta}v(t_n)= e^{i(t-t_0)\Delta}v_0-i\int_{t_0}^{t_n} \Utms g(F+v)ds
= v(t)+i\int_{t_n}^t\Utms g(F+v)ds
\end{align*}
Thus by the continuity of the $\dot{W}$ norm, we find
\begin{align}
&\|e^{i(t-t_n)\Delta} v(t_n)\|_{\dot{W}([t_n,T_+))}+\| F\|_{\dot{W}([t_n,T_+))}\nonumber\\
\leq &\| v\|_{\dot{W}([t_n,T_+))}+C_d(\| v\|_{\dot{W}([t_n,T_+))}^{p}+\| F\|_{\dot{W}([t_n,T_+))}^{p})+\| F\|_{\dot{W}([t_n,T_+))}\leq\f{\epsilon}{2}\label{eqn778}
\end{align}
for $n$ sufficiently large. Then since $F$, $e^{i(t-t_n)\Delta} v(t_n)\in \dot{W}(\R)$ we find $\eta>0$ such that $$\|e^{i(t-t_n)\Delta} v(t_n)\|_{\dot{W}([t_n,T_++\eta])}+\| F\|_{\dot{W}([t_n,T_++\eta])}\leq\epsilon$$
Therefore by the local well-posedness result we can extend the solution to $T_++\eta$, which is a contradiction.

Finally, we turn to scattering. Suppose that $T_+=\infty$ and $\| v\|_{\dot{W}([t_0,\infty))}<\infty$. Define $$v_+=e^{-it_0\Delta}v_0-i\int_{t_0}^\infty e^{-is\Delta}g(F+v)ds$$
Then for any $t>t_0$, the dual Strichartz estimate (\ref{dual}) gives
\begin{align*}
\|v(t)-\Ut v_+\|_{\dot{H}^1(\R^d)}&= \left\| \int_t^\infty \Utms g(F+v)(s)ds\right\|_{\dot{H}^1(\R^d)}\\
&\lesssim_d \|v\|_{\dot{W}([t,\infty))}^{p}+\|F\|_{\dot{W}([t,\infty))}^{p}\rightarrow 0 \text{ as } t\rightarrow \infty
\end{align*}
since $\|v\|_{\dot{W}([t_0,\infty))}<\infty$ and $\|F\|_{\dot{W}([t_0,\infty))}<\infty$. Thus $v_+\in \ho$ and the solution $v$ scatters to $v_+$ as $t\rightarrow +\infty$.
Lastly, the fact that $v$ is the limit of $(\vn)_n$ on compact subintervals of $(T_-,T_+)$ follows by induction of the existence proof over subintervals on which $\|v\|_{\dot{W}}$ and $\|F\|_{\dot{W}}$ are small, using (\ref{eqn778}) to obtain (\ref{F}) on each interval. The number of such intervals required is controlled due to the time-divisibility of the $\dot{W}$-norm.
\end{proof}

\begin{remark}\label{dyadic}
Observe that by applying Strichartz's inequality at dyadic scales followed by the dual estimate (\ref{rtiq}) we obtain
\begin{align*}
\left(\sum_{N\in 2^{\mathbb{Z}}}N^2\|P_Nv\|_{q,r[I]}^2\right)^\frac{1}{2}\lesssim&
\left(\sum_{N\in 2^{\mathbb{Z}}}\|P_Nv_0\|_{\dot{H}^1}^2\right)^\frac{1}{2}+\left(\sum_{N\in 2^{\mathbb{Z}}}\|P_N\na g(F+v)\|_{2,\f{2d}{d+2}[I]}^2\right)^\frac{1}{2}\\
\lesssim&\|v_0\|_{\dot{H}^1}+\|\na g(F+v)\|_{2,\f{2d}{d+2}[I]}<\infty
\end{align*}
for any Strichartz pair $(q,r)$, $I\subset\subset I^*$.
\end{remark}

\section{Conditional Scattering}		\label{scatteringsection}
In this section we will prove the following theorem, giving a sufficient condition for scattering of the solution to the forced NLS
\begin{equation}	\label{dFNLS}
\begin{cases}
(i\dd_t+\Delta)v=|F+v|^{\pmo}(\fv)\\
v(t_0)=v_0\in H^1(\R^d)
\end{cases}
\end{equation}
studied in the previous section.

\begin{theorem}(Conditional Scattering)	\label{conditionalscattering}
Let $v_0\in H^1(\R^d)$, $F\in R\cap L^\infty_t L^{\f{2d}{d-2}}_x(\R)$ (see the definition in \ref{Rdef}). Let $v(t)$ be the solution to (\ref{dFNLS}) defined on its maximal interval of existence $I^*$.  Suppose moreover that $$M:=\sup_{t\in I^*} E(v(t))<\infty$$ 
Then $I^*=\R$, i.e. $v(t)$ is globally defined, and it holds that 
\begin{align}	\label{Wbound}
\|v\|_{\dot{W}(\R)}\leq C(M,\|F\|_{\dot{R}(\R)},d)
\end{align}
As a result, the solution $v$ scatters in $\dot{H}^1$ as $t\rightarrow \pm \infty$.
\end{theorem}

\bigskip
Throughout this section $v$ will refer to the solution to (\ref{dFNLS}) obtained in Theorem \ref{wellposedness}, defined on its maximal interval of existence $I^*:=(T_-,T_+)$.
We first present a lemma bounding the $\dot{W}(\R)$ norm of solutions to the unforced defocusing equation
\begin{equation}	\label{NLS2}
\begin{cases}
(i\dd_t+\Delta)u=|u|^{\pmo}u\\
u(t_0)=u_0\in \dot{H}^1(\R^d)
\end{cases}
\end{equation}

\begin{lemma}		\label{nondec}
There exists a non-decreasing function $K:[0,\infty)\rightarrow [0,\infty)$ with the following property. Let $u_0\in \dot{H}^1(\R^d)$ and $t_0\in \R$. Then there exists a unique global solution $u\in C(\R;\dot{H}^1(\R^d))$ to the defocusing energy-critical NLS (\ref{NLS2}) satisfying $$\|u\|_{\dot{W}(\R)}\leq K(E(u_0))$$ where $$E(u_0):=\frac{1}{2}\intr |\na u_0|^2dx+\frac{d-2}{2d}\intr|u_0|^{\frac{2d}{d-2}} dx$$
\end{lemma}

\begin{proof}
The existence of a global solution follows from the work of Visan \cite{visan2007defocusing}. Combining Theorem 1.1 and Lemma 3.1 of \cite{visan2007defocusing} with the conservation of energy for solutions to (\ref{NLS2}),\footnote{Conservation of energy for solutions to (\ref{NLS2}) is well-known. Nonetheless we remark that, as in the next section of this paper, the formal calculations used to prove it can for example be justified via the regularisation (\ref{reg_nls}), using the stability theory in Theorem 1.3 of \cite{tao2005stability} to show that solutions of the regularised problem converge locally uniformly to solutions of (\ref{NLS2}) in $\dot{H}^1_x$.} we infer the existence of a non-decreasing function $K:[0,\infty)\rightarrow [0,\infty)$ such that the solution $u\in C(\R;\dot{H}^1(\R^d))$ to (\ref{NLS2}) satisfies $$\|u\|_{\dot{S}^1(\R \times \R^d)} \leq K(E(u_0))$$
where $\|u\|_{\dot{S}^1(\R\times\R^d)}:=\sup \left(\sum_{N\in 2^{\Z}} N^2\|P_N u\|_{q,r[\R]}^2\right)^\frac{1}{2}$, with the supremum taken over all Strichartz admissible pairs $(q,r)$. Since this norm controls the $\dot{W}$-norm (by the Littlewood-Paley inequality) we have the result.
\end{proof}

Given the blow-up criterion proved in Theorem \ref{wellposedness}, to prove global existence and scattering of $v$ it is sufficient to show that
\begin{align}		\label{Wbound2}
\|v\|_{\dot{W}(I^*)}<\infty
\end{align}
With this in mind, and in light of Lemma \ref{nondec}, we will develop a suitable perturbation theory to compare solutions of (\ref{dFNLS}) with those of (\ref{NLS2}) in $\dot{W}$. 

We start with a lemma concerning short-time perturbations.

\begin{lemma}(Short-time perturbations)		\label{shorttime}
Let $I{\subset} \R$ be a compact time interval containing $t_0$ and let $u_0, v_0\in H^1(\R^{d})$ with 
\begin{align*}
\|u_0\|_{\ho},\|v_0\|_{\ho}\leq E
\end{align*}
for some $E>0$. Let $u$ solve the defocusing NLS (\ref{NLS2}) with initial data $u(t_0)=u_0$.

Let $F\in R\cap L^{\infty}_t L^\f{2d}{d-2}_x(\R)$. Then there exists a constant $\epsilon_0(E,d)\in (0,1)$ such that if we further suppose
\begin{gather}
\|u\|_{\dot{W}(I)}\leq \epsilon_0		\label{condition1}\\
\left(\sum_{N}\|P_N{\Uto}(u_0-v_0)\|_{\dot{W}(I)}^2\right)^\frac{1}{2}\leq \epsilon		\label{condition2}\\
\|F\|_{\dot{R}(I)}\leq \epsilon		\label{condition3}
\end{gather}
for some $0<\epsilon<\epsilon_0$, then there exists a unique solution $v:I\times \R^d\rightarrow \mathbb{C}$ solving the forced equation (\ref{dFNLS}) with $v(t_0)=v_0$ satisfying
\begin{gather}
\|v-u\|_{\dot{W}(I)}\leq C_{d,1}\epsilon^{\frac{7}{(d-2)^2}}			\label{result1} \\
\|\na [g(F+v)-g(u)]\|_{2,\frac{2d}{d+2}[I]}\leq C_{d,1} \epsilon^{\frac{28}{(d-2)^3}}	\label{difference_bound}
\end{gather}
for a constant $C_{d,1}>1$ depending only on the dimension $d$.
\end{lemma}

\begin{proof}
In view of the local existence theory, it suffices to prove (\ref{result1}) and (\ref{difference_bound}) as a priori estimates
. In what follows all spacetime norms are taken over $I\times \R^d$.
Define $w:=v-u$, which solves
\begin{equation}	\label{weqn}
\begin{cases}
(i\dd_t+\Delta)w=|u+w+F|^{\pmo}(u+w+F)-|u|^{\pmo}u \text{ on } I\times \R^{d} \\
w(t_0)=v_0-u_0
\end{cases}
\end{equation}
with $\|w(t_0)\|_{\dot{H}^1}\lesssim E$. We have
\begin{align}       \label{boundd}
\|w\|_{\dot{W}} \lesssim_d \|e^{i(t-t_0)\Delta} w(t_0)\|_{\dot{W}}+\|\na[g(\fuw)-g(u)]\|_{2,{\frac{2d}{d+2}}}	
\end{align}
where by (\ref{gderiv})
\begin{align*}
&\text{ }\|\na[g(\fuw)-g(u)]\|_{2,{\frac{2d}{d+2}}} \\
\lesssim&_d\||F|^{\pmo}\na F\|_{2,{\frac{2d}{d+2}}}+\||w|^{\pmo}\na F\|_{2,{\frac{2d}{d+2}}}+\||F|^{\pmo}\na u\|_{2,{\frac{2d}
{d+2}}}+\||w|^{\pmo}\na u\|_{2,{\frac{2d}{d+2}}}\\
&+\||F|^{\pmo}\na w\|_{2,{\frac{2d}{d+2}}}+\||w|^{\pmo}\na w\|_{2,{\frac{2d}{d+2}}}+\||u|^{\pmo}
\na F\|_{2,{\frac{2d}{d+2}}}+\||u|^{\pmo}\na w\|_{2,{\frac{2d}{d+2}}}	 
\end{align*}
Then using the nonlinear bound $\||u_1|^{\pmo} \na u_2\|_{2,\frac{2d}{d+2}}\lesssim_d \|u_1\|_{\dot{W}}^{\pmo}\|u_2\|_{\dot{W}}$, assumptions (\ref{condition1})-(\ref{condition3}) and Young's inequality we have
\begin{align}
&\|\na[g(\fuw)-g(u)]\|_{2,{\frac{2d}{d+2}}} \nonumber\\
\lesssim&_d \epsilon^{\p}+\epsilon^{\pmo}\epsilon_0+\epsilon_0^{\pmo}\epsilon+\epsilon^{\pmo}\|w\|_{\dot{W}}+\epsilon_0^{\pmo}\|w\|_{\dot{W}}+\epsilon\|w\|_{\dot{W}}^{\pmo}	\nonumber\\
&+\|w\|_{\dot{W}}^{\p}+\||w|^{\pmo}\na u\|_{2,{\frac{2d}{d+2}}}\nonumber\\
\lesssim&_d \epsilon^{\pmo}+\epsilon_0^{\pmo}\|w\|_{\dot{W}}+\||w|^{\pmo}\na u\|_{2,{\frac{2d}{d+2}}}+\|w\|_{\dot{W}}^{\p} \label{continuity} 
\end{align}
taking $\epsilon_0(d)<1$.

It is tempting to also expand the remaining term in $\dot{W}$ and run a continuity argument, however this will produce the term $\epsilon_0 \|w\|_{\dot{W}}^{\pmo}$ on the right hand side which is an issue for $d>6$ since the power $\pmo$ is less than $1$. We therefore make use of auxiliary spaces $X$ and $Y$ introduced by Visan and Tao in \cite{tao2005stability}. These spaces invoke only $\frac{4}{d+2}<\f{4}{d-2}$ derivatives instead of a whole derivative as in $\dot{W}$ making it possible to run a standard continuity argument in $X$. Once we have a bound on $\|w\|_X$ we can use it in (\ref{continuity}) to bound $\|w\|_{\dot{W}}$.

The spaces $X$ and $Y$ are defined by the following norms:
$$\|f\|_X:=\left(\sum_N N^{\frac{8}{d+2}} \|P_Nf\|_{d+2,\frac{2(d+2)}{d}}^2\right)^\frac{1}{2}$$
$$\|h\|_Y:=\left(\sum_N N^{\frac{8}{d+2}} \|P_Nh\|_{\frac{d+2}{3},\frac{2(d+2)}{d+4}}^2\right)^\frac{1}{2}$$
Note that the space $X$ scales in the same way as $\dot{W}$. Observe also that $\|F\|_X\leq \|F\|_{\dot{R}}$.

To see that $w$ belongs to $X$ we observe the following relation between $X$ and $\dot{W}$:
By Bernstein's inequality we have
\begin{align*}
\|f\|_X\lesssim&_d \left( \sum_N N^2 \|P_N f\|_{d+2,\frac{2d(d+2)}{d^2+2d-4}}^2\right)^\frac{1}{2}
\end{align*}
We now interpolate the $L^{d+2}_tL^{\frac{2d(d+2)}{d^2+2d-4}}_x$ norm between $L^{\frac{2(d+2)}{d-2}}_tL^{\frac{2d(d+2)}{d^2+4}}_x$ and $L^\infty_t L^2_x$ yielding
$$\|P_N  f\|_{d+2,\frac{2d(d+2)}{d^2+2d-4}}\lesssim \|P_N  f\|_{\frac{2(d+2)}{d-2},\frac{2d(d+2)}{d^2+4}}^{\frac{2}{d-2}} \|P_N  f\|_{\infty,2}^{\frac{d-4}{d-2}}$$
and apply H\"{o}lder's inequality for sequences to get
\begin{align}
\|f\|_X\lesssim&_d \left( \sum_N N^2 \|P_N  f\|_{\frac{2(d+2)}{d-2},\frac{2d(d+2)}{d^2+4}}^2\right)^{\frac{2}{2(d-2)}}
\left( \sum_N N^2 \|P_N f\|_{\infty,2}^2\right)^{\frac{d-4}{2(d-2)}}\nonumber\\
\lesssim&_d \left(\sum_N\|P_N f\|_{\dot{W}}^2\right)^{\frac{1}{d-2}} \left( \sum_N \|P_N f\|_{L^{\infty}_t \dot{H}^1_x}^2 \right)^{\frac{d-4}{2(d-2)}}\label{Xbound}
\end{align}
Thus by Remark \ref{dyadic}, $w$ indeed belongs to $X$.

We can use this to bound the remaining term in (\ref{continuity}). Indeed, by the Littlewood-Paley inequality followed by Bernstein's inequality we have
\begin{align*}
\||w|^{\pmo}\na u\|_{2,\frac{2d}{d+2}}\leq & \|w\|_X^{\pmo} \|\na u\|_{\frac{2(d^2-4)}{d^2-12},\frac{2d(d^2-4)}{d^3-2d^2-4d+24}}
\end{align*}
where, since $(\frac{2(d^2-4)}{d^2-12},\frac{2d(d^2-4)}{d^3-2d^2-4d+24})$ is a Strichartz pair, we can use the nonlinear estimate to bound 
\begin{align*}
\|\na u\|_{\frac{2(d^2-4)}{d^2-12},\frac{2d(d^2-4)}{d^3-2d^2-4d+24}}\lesssim_d \|u_0\|_{\ho}+\|u\|_{\dot{W}}^\p\lesssim_d E+\epsilon_0^\p\lesssim_d E
\end{align*}
for $\epsilon_0(E,d)$ sufficiently small.

Substituting this into (\ref{continuity}) and combining the result with (\ref{boundd}) we have
\begin{equation}        \label{bbound}
    \|w\|_{\dot{W}}\lesssim_d \|e^{i(t-t_0)\Delta}w(t_0)\|_{\dot{W}}+\epsilon^{\pmo}+\epsilon_0^{\pmo}\|w\|_{\dot{W}}+E\|w\|_X^{\pmo}+\|w\|_{\dot{W}}^{\p}
\end{equation}

So we must show that $\|w\|_X$ is small. This will require two estimates both proved in \cite{tao2005stability}, see also \cite{exotic}. The first (Lemma 3.2, \cite{tao2005stability}) is a Strichartz-type estimate between $X$ and $Y$:
\begin{equation}	\label{lin}
\left\| \int_{t_0}^t \Utms F(s)ds\right\|_X\lesssim_d \|F\|_Y
\end{equation}
and the second (Lemma 3.3, \cite{tao2005stability}) is the nonlinear estimate 
\begin{equation}	\label{nonlin}
\|g_z(v)u\|_Y\lesssim_d \|v\|_{\dot{W}}^{\pmo}\|u\|_X
\end{equation}
(with a similar estimate for $g_{\bar{z}}$).

Using (\ref{lin}) and the fact that $w$ satisfies equation (\ref{weqn}) we immediately obtain
\begin{equation}	\label{wbound}
\|w\|_X\lesssim_d \|e^{i(t-t_0)\Delta}w(t_0)\|_X+\|g(\fuw)-g(u)\|_Y
\end{equation}
First consider the free evolution term. By (\ref{Xbound}) we see
\begin{align*}
\|e^{i(t-t_0)\Delta} w(t_0)\|_X\lesssim&_d \left(\sum_N\|P_N e^{i(t-t_0)\Delta} w(t_0)\|_{\dot{W}}^2\right)^{\frac{1}{d-2}} \left( \sum_N \|P_N w(t_0)\|_{\dot{H}^1}^2 \right)^{\frac{d-4}{2(d-2)}}\\
\lesssim&_d \epsilon^{\frac{2}{d-2}}E^{\frac{d-4}{d-2}}
\end{align*}

We now move onto the second term in (\ref{wbound}). Using (\ref{gdiff}), Minkowski's inequality and the nonlinear estimate (\ref{nonlin}) we have
\begin{align*}
\|g(\fuw)-g(u)\|_Y\leq& \int_0^1 \|u+\theta(F+w)\|_{\dot{W}}^{\pmo}\|F+w\|_X d\theta\\
\lesssim&_d (\|u\|_{\dot{W}}^{\pmo}+\|F\|_{\dot{W}}^{\pmo}+\|v\|_{\dot{W}}^{\pmo})(\|F\|_X+\|w\|_X)\\
\lesssim&_d (\epsilon_0^{\pmo}+\epsilon^{\pmo}+\|v\|_{\dot{W}}^{\pmo})(\epsilon+\|w\|_X)
\end{align*}
where we used that $\|F\|_{\dot{W}}+\|F\|_X= \|F\|_{\dot{R}}\leq \epsilon$ in the last line.

To bound $\|v\|_{\dot{W}}$, we first use that $v_0$ is close to $u_0$ and that $u$ satisfies the standard NLS (\ref{NLS2}) to bound the linear part:
\begin{align*}
\|\Uto v_0\|_{\dot{W}}\lesssim&_d \|\Uto (v_0-u_0)\|_{\dot{W}}+\|\Uto u_0\|_{\dot{W}}\\
\lesssim&_d \epsilon+\|u\|_{\dot{W}}+\|u\|_{\dot{W}}^{\p}\\
\lesssim&_d \epsilon_0
\end{align*}
We can thus apply the local well-posedness theory to infer that, on the interval $I$, $$\|v\|_{\dot{W}}\lesssim_d \epsilon_0$$

Returning to (\ref{wbound}) we thus have 
\begin{align*}
\|w\|_X\lesssim&_d \epsilon^{\frac{2}{d-2}}E^{\frac{d-4}{d-2}}+(\epsilon_0^{\pmo}+\epsilon^{\pmo}+\epsilon_0^{\pmo})(\epsilon+\|w\|_X)\\
\lesssim&_d \epsilon^{\frac{2}{d-2}}E^{\frac{d-4}{d-2}}+\epsilon+\epsilon_0^{\pmo}\|w\|_X
\end{align*}
Thus choosing $\epsilon_0$ sufficiently small depending on $d$ and $E$ we conclude $$\|w\|_X\lesssim_d \epsilon^{\frac{2}{d-2}}E^{\frac{d-4}{d-2}}$$

Now that we have bounded $\|w\|_X$, we can return to (\ref{bbound}) to bound $\|w\|_{\dot{W}}$. We have
\begin{align*}
\|w\|_{\dot{W}}\lesssim&_d \|\Uto w(t_0)\|_{\dot{W}}+\epsilon^{\pmo}+\epsilon_0^{\pmo}\|w\|_{\dot{W}}+E\|w\|_X^{\pmo}+\|w\|_{\dot{W}}^{\p}\\
\lesssim&_d \epsilon+\epsilon^{\pmo}+\epsilon_0^{\pmo}\|w\|_{\dot{W}}+E(\epsilon^{\frac{2}{d-2}}E^{\frac{d-4}{d-2}})^{\pmo}+\|w\|_{\dot{W}}^{\p}\\
\lesssim&_d \epsilon^{\frac{7}{(d-2)^2}}+\frac{1}{2}\|w\|_{\dot{W}}+\|w\|_{\dot{W}}^{\p}
\end{align*}
for $\epsilon_0(E,d)$ sufficiently small. 
The result (\ref{result1}) now follows from a standard continuity argument.

Lastly, we show (\ref{difference_bound}).
By (\ref{continuity})
\begin{align*}
\|\na [g(F+v)-g(u)]\|_{2,\frac{2d}{d+2}}\lesssim&_d \epsilon^{\pmo}+\epsilon_0^{\pmo}\|w\|_{\dot{W}}+\|u\|_{\dot{W}}\|w\|_{\dot{W}}^{\pmo}+\|w\|_{\dot{W}}^{\p}\\
\lesssim&_d \epsilon^{\pmo}+\|w\|_{\dot{W}}+\|w\|_{\dot{W}}^{\pmo}+\|w\|_{\dot{W}}^{\p}
\end{align*}
using the bounds assumed on $\|u\|_{\dot{W}}$ and $\|F\|_{\dot{W}}$. Substituting in the bound just obtained for $\|w\|_{\dot{W}}$ gives the result.
\end{proof}

We now extend this result by removing the smallness assumption on $u$ in the case when $u$ and $v$ have the same initial data.
\begin{lemma}(Long-time perturbations)	\label{longtime}
Let $I{\subset} \R$ be a compact time interval with $t_0\in I$ and let $v_0\in H^1(\R^d)$ with $$E(v_0)\leq E$$
 Let $u\in C(I,\dot{H}^1(\R^d))$ be the solution to the defocusing NLS (\ref{NLS2}) with initial data $u(t_0)=v_0$ and $$\|u\|_{\WI}\leq K$$ for some $K>0$. Then there exists $\epsilon_1(E,K,d) \in (0,1)$ such that for any $F\in R(I)$ sufficiently small in the sense that $$\|F\|_{\dot{R}(I)}\leq \epsilon_1,$$ there exists a unique solution $v:I\times \R^{d}\rightarrow \mathbb{C}$ to the forced equation (\ref{dFNLS}) with initial data $v(t_0)=v_0$ and it holds 
\begin{equation}
\|v-u\|_{\WI}\leq 1		\label{result_long}
\end{equation}
\end{lemma}

\begin{proof}
Without loss of generality assume $t_0=\inf I$. As in the previous proposition it suffices to prove the bound as an a priori estimate.
In order to make use of the short-time perturbation theory, we will induct over intervals on which the $\WI$-norm of $u$ is small. To this end we partition $I$ into consecutive intervals with disjoint interiors $(I_j)_{j=1}^J$ such that
\begin{equation}		\label{usmall}
\|u\|_{\dot{W}(I_j)}\leq \epsilon_0(2(2E)^{\frac{1}{2}},d)
\end{equation}
for each $j=1,...,J$, where $\epsilon_0$ is as in Lemma \ref{shorttime}. By the time-divisibility properties of $\dot{W}$ we are able to do this with 
\begin{equation}
J\lesssim \left(\frac{K}{\epsilon_0(2(2E)^{\frac{1}{2}},d)}\right)^{\frac{2(d+2)}{d-2}}
\end{equation}

Denote $I_j=[t_{j-1},t_j]$. We must check that the conditions of Lemma \ref{shorttime} are satisfied on this interval.

We first make two observations. Using Strichartz's inequality (\ref{dual}) we have
\begin{align}
\|v(t_{j})-u(t_j) \|_{\ho}\simeq& \left \| \int_{t_0}^{t_j}e^{i(t_j-s)\Delta}\na [g(F+v)-g(u)](s)ds \right\|_{L^2(\R^d)}		\nonumber\\
\leq& A_{d,1} \|\na[g(F+v)-g(u)]\|_{2,\f{2d}{d+2} [t_0,t_j]}		\label{hone}
\end{align}
\begin{sloppypar}
Secondly, by the Strichartz estimates (\ref{hom}) and (\ref{dual}) followed by the embedding (\ref{rtiq}) we obtain
\end{sloppypar}
\begin{align}
&\left(\sum_N\|P_Ne^{i(t-t_j)\Delta}(v(t_j) -u(t_j) )\|_{\dot{W}(I_{j+1})}^2\right)^\frac{1}{2}		\nonumber\\
&\simeq \left( \sum_N N^2 \left\| \int_{t_0}^{t_j} \Utms P_N[g(F+v)-g(u)](s) ds \right\|_{\frac{2(d+2)}{d-2},\frac{2d(d+2)}{d^2+4}[I_{j+1}]}^2\right)^\frac{1}{2}		\nonumber\\
&\leq A_{d,2} \left(\sum_N N^2 \|P_N[g(F+v)-g(u)](s)\|_{2,\f{2d}{d+2}[t_0,t_j]}^2\right)^\frac{1}{2}		\nonumber\\
&\leq A'_{d,2} \|\na [g(F+v)-g(u)]\|_{2,\f{2d}{d+2}[t_0,t_j]}		\label{w}
\end{align}
Set $A_d:=\max\{1,A_{d,1},A'_{d,2}\}$.

We now prove a technical claim that will be useful for the inductive step. In the rest of this proof we denote $\alpha:=\frac{28}{(d-2)^3}$ and $C_d:=A_d C_{d,1}\geq 1$, with $C_{d,1}$ the constant from Lemma \ref{shorttime}.
\begin{claim}
We may take $\epsilon_1(E,K,d)>0$ sufficiently small such that the following holds:
Define a sequence $(\epsilon(j))_{j=1}^{J+1}$ by 
\begin{align*}
\epsilon(1)=\epsilon_1(E,K,d), &&
\epsilon(j+1)=C_d \sum_{k=1}^j \epsilon(k)^{\alpha} \text{ for } 1\leq j \leq J
\end{align*}
Then for all $1\leq j\leq J+1$ it holds
\begin{equation*}
\epsilon_1\leq\epsilon(j)\leq (2C_d)^{ \sum_{k=0}^{j-2}  \alpha ^k } \epsilon_1^{\alpha^{j-1}}< \min\{\epsilon_0(2(2E)^{\frac{1}{2}},d),(2E)^{\frac{1}{2}}\}
\end{equation*}
\end{claim}

\begin{proof}[Proof of claim.]
The cases $j=1$, $2$ are easily verified (since $\alpha<1$).
Suppose that the claim holds for all $1\leq j\leq k$ for some $k\leq J$. Then by definition
\begin{align*}
\epsilon(k+1)=\epsilon(k)+C_d \epsilon(k)^\alpha\leq 2C_d \epsilon(k)^\alpha\leq (2C_d)^{\sum_{k'=0}^{k-1}\alpha^{k'}} \epsilon_1^{\alpha^k}
\end{align*}
as required. That $\epsilon(k)\geq \epsilon_1$ is clear since the sequence is increasing, and for $1\leq k \leq J$,
\begin{align*}
\epsilon(k+1)\leq& (2C_d)^{\sum_{k'=0}^{J-1}\alpha^{k'}} \epsilon_1^{\alpha^{J}}< \min\{\epsilon_0(2(2E)^{\frac{1}{2}},d),(2E)^{\frac{1}{2}}\}
\end{align*}
for $\epsilon_1$ sufficiently small depending on $E$, $K$ and $d$.
\end{proof}

We now prove a second claim in which we reduce the long time perturbation result to the short time result on the intervals $I_j$. In the rest of this proof we will take $\epsilon_1(E,K,d)$ as in the above claim.
\begin{claim}
Under the assumptions of the lemma, for all $1\leq j\leq J$ it holds
\begin{gather*}
\|v(t_{j-1}) \|_{\ho}\leq2(2E)^{\frac{1}{2}}\\
\left(\sum_N\|P_Ne^{i(t-t_{j-1})\Delta}(\ujmo-v(t_{j-1}) )\|_{\dot{W}(I_j)}^2\right)^{\frac{1}{2}} \leq \epsilon(j)< \epsilon_0(2(2E)^{\frac{1}{2}},d)\\
\|\na [g(F+v)-g(u)]\|_{2,\f{2d}{d+2}[I_j]}\leq C_{d,1}{\epsilon(j)}^\alpha\\
\|u-v\|_{\dot{W}(I_j)}\leq C_{d,1} \epsilon(j)^{\frac{7}{(d-2)^2}}
\end{gather*}
for $\epsilon(j)$ as in the previous claim.
\end{claim}
\begin{proof}[Proof of claim]
Recall that for each $j=1,...,J$ it holds
\begin{gather*}
\|u\|_{\dot{W}(I_j)}\leq \epsilon_0(2(2E)^{\frac{1}{2}},d)\\
\|F\|_{\dot{R}(I_j)}\leq \epsilon_1< \epsilon_0(2(2E)^{\frac{1}{2}},d)\\
\|u(t_{j-1})\|_{\dot{H}^1}\leq (2E)^{\frac{1}{2}}
\end{gather*}
where we used that $u$ has conserved energy.

For $j=1$, we have $u(t_0)=v(t_0)=v_0$ so we can immediately apply Lemma \ref{shorttime} to obtain (using $\epsilon(1)=\epsilon_1$) $$\|\na [g(F+v)-g(u)]\|_{\f{2d}{d+2}[I_1]}\leq C_{d,1}\epsilon(1)^\alpha$$ and $$\|u-v\|_{\dot{W}(I_1)}\leq C_{d,1} \epsilon(1)^{\frac{7}{(d-2)^2}}$$
Now suppose that the claim holds for all $1\leq j\leq k$ for some $k\leq J-1$. Then by (\ref{hone}) we have
\begin{align*}
\|u(t_k) -v(t_k) \|_{\ho}\leq& A_d   \sum_{k'=1}^k\|\na [g(F+v)-g(u)]\|_{2,\f{2d}{d+2}[I_{k'}]}\\
\leq& A_d\sum_{k'=1}^k C_{d,1} \epsilon(k')^\alpha=\epsilon(k+1)< (2E)^{\frac{1}{2}}
\end{align*}
and so $\|v(t_k)\|_{\dot{H}^1(\R^d)}< 2(2E)^\frac{1}{2}$.
Similarly using (\ref{w}) we see that
\begin{align*}
\left(\sum_N\|P_Ne^{i(t-t_k)\Delta}(v(t_k) -u(t_k) )\|_{\dot{W}(I_{k+1})}^2\right)^\frac{1}{2}\leq& A_d \sum_{k'=1}^k\|\na [g(F+v)-g(u)]\|_{2,\f{2d}{d+2}[I_{k'}]}\\
\leq& \epsilon(k+1)< \epsilon_0(2(2E)^{\frac{1}{2}},d)
\end{align*}
Thus since also $\|F\|_{\dot{R}(I_{k+1})}\leq \epsilon_1 \leq \epsilon(k+1)$, we can apply Lemma \ref{shorttime} on $I_{k+1}$ to obtain
\begin{equation*}
\|u-v\|_{\dot{W}(I_{k+1})}\leq C_{d,1} \epsilon(k+1)^{\frac{7}{(d-2)^2}}
\end{equation*}
and
\begin{equation*}
\|\na [g(F+v)-g(u)]\|_{2,\f{2d}{d+2}[I_{k+1}]}\leq C_{d,1} \epsilon(k+1)^\alpha
\end{equation*}
This completes the proof of the claim.
\end{proof}

We now sum the bounds over all the sub-intervals and use that $\alpha< \frac{7}{(d-2)^2}$ to obtain
\begin{align*}
\|u-v\|_{\dot{W}(I)}\leq C_{d,1} \sum_{j=1}^J \epsilon(j)^{\alpha}\leq \epsilon(J+1)< 1
\end{align*}
\end{proof}

Using the perturbation theory developed we are now able to prove the conditional scattering result.
\begin{proof}[Proof of Theorem \ref{conditionalscattering}]
By the local well-posedness theory, it remains to prove that $$\|v\|_{\dot{W}(I^*)}\leq C(M,\|F\|_{\dot{R}(\R)},d)$$
Consider first $[t_0,T_+)$. Partition $[t_0,T_+)$ into $J$ consecutive intervals $I_j$ such that $$\|F\|_{\dot{R}(I_j)}\leq \epsilon_1(M,K(M),d)$$ where $K$ is the non-decreasing function from Lemma \ref{nondec} and $\epsilon_1$ is the constant from Lemma \ref{longtime}. Due to the time divisibility of the $\dot{R}$-norm, we can do this with $$J\lesssim \left( \frac{\|F\|_{\dot{R}(\R)}}{\epsilon_1(M,K(M),d)}\right)^{d+2}$$ 

Denote $I_j=[t_{j-1},t_j]$ for $1\leq j \leq J$. On each $I_j$ we compare $v$ with the solution $u_j$ to the usual defocusing NLS (\ref{NLS2}) with initial data $u(t_{j-1})=v(t_{j-1})$, satisfying $E(u(t_{j-1}))\leq M$ by assumption.
By Lemma \ref{nondec} we know that such a solution $u_j$ exists globally in time and satisfies $$\|u_j\|_{\dot{W}(I_j)}\leq \|u_j\|_{\dot{W}(\R)}\leq K(M)$$
We can therefore apply Lemma \ref{longtime} on each $I_j$ to see that $v$ satisfies $$\|v\|_{\dot{W}(I_j)}\leq \|v-u_j\|_{\dot{W}(I_j)}+\|u_j\|_{\dot{W}(I_j)}\leq K(M)+1$$
Summing the estimates over the intervals $I_j$ and arguing in the same way on $(T_-,t_0]$ implies the result.
\end{proof}

\section{Energy bound}		\label{energybound}
In this section we prove that the solution $v$ to the forced NLS \eqref{dFNLS} does indeed satisfy a uniform in time energy bound on its maximal interval of existence $I^*=(T_-,T_+)$, under assumptions on $F$ which we shall prove to hold almost surely for $F=\Ut \fomega$ in the next section. The precise result is the following.

\begin{theorem}\label{thm334}
Let $v_0\in H^1(\R^d)$, $t_0=0$. Denote by $v\in C^0_t H^1_x\medcap W(I^*)$ the unique solution to \eqref{dFNLS} obtained in Theorem \ref{wellposedness}, for a forcing term $F\in R\cap L^\infty_t L^{\f{2d}{d-2}}_x(\R)$ which solves the linear Schr\"{o}dinger equation $(i\dd_t+\Delta)F=0$ with $L^2$ initial data and satisfies
\begin{equation*}
F\in L^{\f{1}{\sigma }}_t L^{\f{2d}{d-4\sigma }}_x
\end{equation*}
\begin{equation*}
\na F\in L^2_tL^{\frac{4d-2}{2d-3-\sigma }}_x \medcap L^2_t L_x^{\f{2d(2d-1)}{2d^2-7d+4+d\sigma }}\medcap L^1_t L^\f{2d}{d-4}_x
\medcap L_t^{\f{d-2}{d-2-4\sigma }}L^{\f{2d(d-2)}{d(d-6)+16\sigma }}_x(\R)	
\end{equation*}
for some $\sigma (d)$ sufficiently small. 
Then it holds
\begin{small}
\begin{align}\label{big_eqn00}
\sup_{t\in I^*} E(v(t))&\leq (1+E(v_0)+\|F\|_{\infty,\pp[\R]}^{\f{2(d-2)}{d-6}}+\|F\|_{\infty,\f{2d}{d-2}[\R]}^{\f{2(d+2)}{d-2}}) \nonumber\\
&\quad\cdot\exp (C_d(\|F\|_{\f{1}{\sigma },\f{2d}{d-4\sigma }[\R]}^{\f{1}{\sigma }}+\|\na F\|^2_{2,\f{4d-2}{2d-3-\sigma }[\R]}+\|\na F\|^2_{2,\f{2d(2d-1)}{2d^2-7d+4+d\sigma }[\R]}\nonumber\\
&\quad\quad\quad\quad\quad\quad\quad\quad\quad\quad\quad
+\|\na F\|_{1,\frac{2d}{d-4}[\R]}+\|\na F\|_{\f{d-2}{d-2-4\sigma },\f{2d(d-2)}{d(d-6)+16\sigma }[\R]}^\f{d-2}{d-2-4\sigma }))
\end{align}
\end{small}
for some $C_d>0$.
\end{theorem}

This will follow from an analogous theorem for the regularised solutions $\vn$ to
\begin{equation}\label{FNLSn3}
\begin{cases}
(i\dd_t+\Delta)v_n=g_n(F_n+v_n)\\
v_n(0)=v_{n,0}\in H^2(\R^d)
\end{cases}
\end{equation}
with $g_n(u)=u\vp_n'(|u|^2)$, $F_n=P_{\leq n}F$ and $v_{n,0}=P_{\leq n}v_0$, as in Section \ref{regularisationsection}.
\begin{theorem}		\label{energyboundthm}
Suppose that $F$ satisfies the assumptions of Theorem \ref{thm334}. 
Let $\vn$ be the unique global solution to \eqref{FNLSn3} in $C(\R,H^2(\R^d))\medcap C^1(\R,L^2(\R^d))$. Then it holds
\begin{sloppypar}
\end{sloppypar}
\vspace{-0.9em}
\begin{small}
\begin{align}
\sup_{t\in \R} E_n(v_n(t))&\leq (1+E(v_{n,0})+\|F\|_{\infty,\pp[\R]}^{\f{2(d-2)}{d-6}}+\|F\|_{\infty,\f{2d}{d-2}[\R]}^{\f{2(d+2)}{d-2}}) \nonumber\\
&\quad\cdot\exp (C_d(\|F\|_{\f{1}{\sigma },\f{2d}{d-4\sigma }[\R]}^{\f{1}{\sigma }}+\|\na F\|^2_{2,\f{4d-2}{2d-3-\sigma }[\R]}+\|\na F\|^2_{2,\f{2d(2d-1)}{2d^2-7d+4+d\sigma }[\R]}\nonumber\\
&\quad\quad\quad\quad\quad\quad\quad\quad\quad\quad\quad
+\|\na F\|_{1,\frac{2d}{d-4}[\R]}+\|\na F\|_{\f{d-2}{d-2-4\sigma },\f{2d(d-2)}{d(d-6)+16\sigma }[\R]}^\f{d-2}{d-2-4\sigma }))\label{big_eqn22}
\end{align}
\end{small}
for some $C_d>0$.
Here $$E_n(\vn(t)):=\frac{1}{2}\intr|\na \vn|^2 dx+\frac{1}{2}\intr\vp_n(|\vn|^2)dx$$
\end{theorem}

Before proving this theorem, we show how it can be used to deduce Theorem \ref{thm334}.
\begin{proof}[Proof of Theorem \ref{thm334} given Theorem \ref{energyboundthm}.]
Fix a compact subinterval $0\in I\subset I^*$. Observe that $E(v_{n,0})\rightarrow E(v_0)$. Therefore, denoting $M_n$ the quantity on the right hand side of inequality \eqref{big_eqn22} and $M$ the right hand side of \eqref{big_eqn00}, we have $M_n\rightarrow M$.

Consider the sequence $v_{n,n}:=v_n\mathds{1}_{|\vn|\leq n}$. 
For every $t\in I$ we have
$$\|v_{n,n}(t)\|_{L^{p+1}(\R^d)}^{p+1}=\int_{|\vn|\leq n}|\vn(t)|^{p+1}dx\leq\f{p+1}{2}\intr\vp_n(|\vn(t)|^2)dx\leq(p+1) \sup_{n'} M_{n'}$$
and so for each $t\in I$ there exists a subsequence $(v_{n_j,n_j})_j$ weakly converging to $v$ in $L^{p+1}(\R^d)$, from which we deduce
$$\|v(t)\|_{L^{p+1}(\R^d)}^{p+1}\leq \liminf_j\f{p+1}{2}\intr\vp_{n_j}(|v_{n_j}(t)|^2)dx$$
for every $t\in I$. Similarly, we have $\|v(t)\|_{\dot{H}^1(\R^d)}^2\leq\liminf_j \|v_{n_j}\|_{\dot{H}^1}^2$, so
$$
E(v(t))\leq\liminf_jE_{n_j}(v_{n_j}(t))\leq\liminf_jM_{n_j}=M
$$
\end{proof}

We now prove Theorem \ref{energyboundthm}. The idea of the proof is to work on small intervals on which $F$ is small and use a bootstrap argument to control the energy increment there. By only placing $F$ into spaces with finite time exponents we are able to iterate this finitely many times to obtain a bound over the whole interval. In early papers on this topic \cite{killip2019almost,dodson2019almost,dodson2020almost}, a double bootstrap method was used, simultaneously controlling the solution in weighted $L^p$ spaces via Morawetz inequalities. Thanks to a randomised $L^1_t$-estimate introduced by Spitz in \cite{spitz2021almost} (see Section \ref{l1t}), this is not necessary here and we can directly bound the energy increment by placing $F$ into spaces of low time-integrability.
 
\begin{proof}[Proof of Theorem \ref{energyboundthm}.]
In this proof we will often use the notation $p$ instead of $\f{d+2}{d-2}$, so any implicit constants depending on $p$ in fact depend only on $d$.
We will show the bound holds on the compact interval $[0,T]$ for any $T>0$. Since the bound does not depend on $T$, this clearly extends to $[0,\infty)$, and the argument in the reverse time direction is analogous.

Define the norm
\begin{multline*}
\|F\|_{Z}:=\|F\|_{\f{1}{\sigma },\f{2d}{d-4\sigma }}+\|\na F\|_{2,\f{4d-2}{2d-3-\sigma }}+\|\na F\|_{2,\f{2d(2d-1)}{2d^2-7d+4+d\sigma }}\\
+\|\na F\|_{1,\frac{2d}{d-4}}+\|\na F\|_{\f{d-2}{d-2-4\sigma },\f{2d(d-2)}{d(d-6)+16\sigma }}
\end{multline*}
on any time interval.
Partition $[0,T]$ into $J$ consecutive subintervals $I_j:=[t_{j-1},t_j]$, $j=1,\ldots,J$ such that $$\|F\|_{Z[I_j]}\leq \eta$$ for each $j=1,\ldots,J$ and some $\eta<1$ to be determined which depends only on the dimension $d$. Note that by the time-divisibility properties of the $Z$ norm, it is possible to do this with 
\begin{align}
J&\lesssim_d \|F\|_{\f{1}{\sigma },\f{2d}{d-4\sigma }[\R]}^{\f{1}{\sigma }}+\|\na F\|^2_{2,\f{4d-2}{2d-3-\sigma }[\R]}\nonumber\\
&\quad+\|\na F\|^2_{2,\f{2d(2d-1)}{2d^2-7d+4+d\sigma }[\R]}+\|\na F\|_{1,\frac{2d}{d-4}[\R]}+\|\na F\|_{\f{d-2}{d-2-4\sigma },\f{2d(d-2)}{d(d-6)+16\sigma }[\R]}^\f{d-2}{d-2-4\sigma }\label{Jbound}
\end{align}
For each $j=1,\ldots,J$ define
$$A_{t_{j-1}}(t_j):=1+\sup_{t\in[t_{j-1},t_{j}]}E_n(v_n(t))$$
In the following calculations all spacetime norms above are taken over $[t_{j-1},t]\times \R^d$. Since $\vn\in C^1(\R,L^2(\R^d))\medcap C(\R,H^2(\R^d))$ one may differentiate $E_n(\vn(t))$ to obtain
$$
\dd_t E_n(\vn(t))=-\re\int_{\R^d}\dd_t\bar{v}_n(\Delta v_n -\vn\varphi_n'(|\vn|^2))dx
$$
which is well-defined since $\dd_t \vn$, $\Delta \vn\in L^2(\R^d)$.
Integrating this over $[t_{j-1},t]$ and performing a calculation similar to that in \cite{killip2019almost,dodson2019almost} (see Appendix \ref{appendix0} for details), we obtain
\begin{align}
&|E_n(v_n(t))-E_n(v_n(\tjmo))|\nonumber\\
&\leq\frac{1}{2}\sup_{[t_{j-1},t]}\|\vp_n(|\fn+\vn|^2)-\vp_n(|\vn|^2)-\vp_n(|\fn|^2)\|_{L^1(\rd)}\label{oneone}\\
&\quad+\|\na\overline{\fn}\cdot\na(g_n(\fn+\vn)-g_n(\fn))\|_{1,1}\label{twotwo}
\end{align}

First consider \eqref{oneone}. Observe that 
$$\left|\vp_n(|\fn+\vn|^2)-\vp_n(|\vn|^2)-\vp_n(|\fn|^2)\right|\lesssim_p|F_n|^p|v_n|+|F_n||v_n|^p$$
uniformly in $n$. Hence by Young's inequality we have
\begin{align*}
&\left\|\vp_n(|\fn+\vn|^2)-\vp_n(|\vn|^2)-\vp_n(|\fn|^2)\right\|_{\infty,1}\\
&\lesssim_d \dl \|v_n\|_{\infty,p+1}^2+C_{\dl,d} \|F_n\|_{\infty,p+1}^{\f{2}{2-p}}+C_{\dl,d} \|F_n\|_{\infty,p+1}^{2p}\\
&\lesssim_d  \dl A_{\tjmo}(t_j )+C_{\dl,d}\|F_n\|_{\infty,p+1}^{\f{2}{2-p}}+C_{\dl,d} \|F_n\|_{\infty,p+1}^{2p}
\end{align*}
for any $\dl>0$, since $\|\vn\|_{\infty,p+1}^2\lesssim\|\vn\|_{L^\infty_t\dot{H}^1_x}^2$.

We now turn to \eqref{twotwo}. This time we use \eqref{gderiv} to bound
\begin{align*}
\eqref{twotwo}\lesssim&_d \||\vn|^{p-1}|\na F_n|^2\|_{1,1}+\||\vn|^{p-1}|\na \vn||\na \fn|\|_{1,1}+\||\fn|^{p-1}|\na \fn||\na\vn|\|_{1,1}
\end{align*}
We can control these terms as follows, using that $A_{t_{j-1}}(t_j)\geq 1$:
\begin{align*}
\||{\vn}|^{p-1}|\na {\fn}|^2\|_{1,1}\leq& \|{\vn}\|_{\infty,p+1}^{p-1}\|\na {\fn}\|_{2,\f{4d-2}{2d-3-\sigma }}\|\na {\fn}\|_{2,\f{2d(2d-1)}{2d^2-7d+4+d\sigma }}\\
\lesssim&_d A_{t_{j-1}}(t_j)\|\na {\fn}\|_{2,\f{4d-2}{2d-3-\sigma }}\|\na {\fn}\|_{2,\f{2d(2d-1)}{2d^2-7d+4+d\sigma }}\\
\||{\vn}|^{p-1} |\na {\vn}| |\na {\fn}|\|_{1,1}\leq& \|{\vn}\|_{\infty,p+1}^{p-1}\|\na {\vn}\|_{\infty,2}\|\na {\fn}\|_{1,\frac{2d}{d-4}}\\
\lesssim&_d A_{t_{j-1}}(t_{j})\|\na {\fn}\|_{1,\frac{2d}{d-4}}\\
\||{\fn}|^{p-1}|\na \fn||\na \vn|\|_{1,1}\leq& \|\na {\vn}\|_{\infty,2}\|{\fn}\|_{\f{1}{\sigma },\f{2d}{d-4\sigma }}^{p-1}\|\na {\fn}\|_{\f{d-2}{d-2-4\sigma },\frac{2d(d-2)}{d(d-6)+16\sigma }}\\
\leq& A_{t_{j-1}}(t_{j})\|{\fn}\|_{\f{1}{\sigma },\f{2d}{d-4\sigma }}^{p-1}\|\na {\fn}\|_{\f{d-2}{d-2-4\sigma },\frac{2d(d-2)}{d(d-6)+16\sigma }}
\end{align*}

Noting that the spaces into which $F_n$ has been placed here are exactly those which make up the $Z$-norm, we can bound each term by $C_d A_{t_{j-1}}(t_j)\eta$, and so by \eqref{oneone}-\eqref{twotwo} it holds
\begin{multline*}
A_{t_{j-1}}(t_j)\lesssim_d 1+E_n(v_n(t_{j-1}))+\dl A_{\tjmo}(t_j )+C_{\dl,d}\|F_n\|_{\infty,p+1[t_{j-1},t_j]}^{\f{2}{2-p}}\\
+C_{\dl,d} \|F_n\|_{\infty,p+1[t_{j-1},t_j]}^{2p}+\eta A_{t_{j-1}}(t_j)
\end{multline*}
Choosing $\dl(d)$ and $\eta(d)$ sufficiently small and using that $\|\fn\|_{a,b}\lesssim\|F\|_{a,b}$ for $1\leq a,b\leq \infty$ (and likewise for $\na \fn$), we thus have
\begin{equation*}
A_{t_{j-1}}(t_j)\leq C_d^*(1+E_n(v_n(t_{j-1}))+\|F\|_{\infty,p+1[\R]}^{\f{2}{2-p}}+\|F\|_{\infty,p+1[\R]}^{2p})
\end{equation*}
for some constant $C_d^*>1$.

Iterating the results on the consecutive intervals $(I_j)_{j=1}^J$, we obtain
$$A_{t_{j-1}}(t_j)\leq(2C_d^*)^{j}(1+E_n(v_n(0))+\|F\|_{\infty,p+1[\R]}^{\f{2}{2-p}}+\|F\|_{\infty,p+1[\R]}^{2p})$$
for all $j=1,\ldots,J$, from which
\begin{equation*}
A_{0}(T)\leq (2C_d^*)^J(1+E(v_{n,0})+\|F\|_{\infty,p+1[\R]}^{\f{2}{2-p}}+\|F\|_{\infty,p+1[\R]}^{2p})
\end{equation*}
where we used that $E_n(v_{n,0})\leq E(v_{n,0})$. Combining this with \eqref{Jbound} yields the result.
\end{proof}

\section{Almost sure bounds for the forcing term}	\label{bounds_section}\label{probsection}
In this section we show that the randomised linear evolution $F^\om:=\Ut \fomega$ almost surely satisfies the conditions required for well-posedness and scattering provided the initial data $f$ lies in a Sobolev space of sufficiently high regularity. In particular, we prove the following theorem:
\begin{theorem}     \label{random_thm}
Let $\max\{\frac{4d-1}{3(2d-1)},\f{d^2+6d-4}{(2d-1)(d+2)}\}< s<1$ and $f\in H^s(\R^d)$. Let $\fomega$ denote the randomisation of $f$ as in \eqref{rand_f} and $F^\om:=\Ut \fomega$. Then when $\sigma (d)$ is sufficiently small we have
\begin{equation*}
F^\om\in L^\infty_tL^{\pp}_x \medcap L^{\f{1}{\sigma }}_t L^{\f{2d}{d-4\sigma }}_x \medcap R(\R)
\end{equation*}
\begin{equation*}
\na F^\om\in L^2_tL^{\frac{4d-2}{2d-3-\sigma }}_x \medcap L^2_t L_x^{\f{2d(2d-1)}{2d^2-7d+4+d\sigma }}\medcap L^1_t L^\f{2d}{d-4}_x
\medcap L_t^{\f{d-2}{d-2-4\sigma }}L^{\f{2d(d-2)}{d(d-6)+16\sigma }}_x(\R)	
\end{equation*}
for almost every $\om\in\Om$.
\end{theorem}
When combined with the results of the previous section, this completes the proof of Theorem \ref{thm1}.

The proof of the bounds in the above theorem is split into subsections according to the method used to obtain the almost sure bound. Throughout, we shall make repeated use of the following important generalisation of Khintchine's inequality due to Burq and Tzvetkov \cite{burq2008random}, formulated here as in \cite{burq2008random}.
\begin{lemma}(Large Deviation Estimate, Lemma 3.1 \cite{burq2008random}) \label{large_deviation}
Let $(g_k)_{k\in \mathbb{N}}$ be a sequence of independent, real-valued, zero-mean random variables on a probability space $(\Omega, \mathcal{A},\mathbb{P})$ with distributions $(\mu_k)_k$ satisfying $$\int_\R e^{\gamma x} d\mu_k(x)\leq e^{c \gamma^2} \quad\quad\quad \forall \gamma\in \R$$ with the constant $c>0$ independent of $k,\gamma$. Then there is a constant $C>0$ such that 
\begin{align*}
\left\| \sum_{k\in \mathbb{N}} c_k g_k\right\|_{L^\beta (\Omega)}\leq C \sqrt{\beta} \left( \sum_{k\in \mathbb{N}} |c_k|^2 \right)^\frac{1}{2}
\end{align*}
for all $(c_k)_k \in \ell^2(\mathbb{N})$ and $\beta \in [2,\infty)$.
\end{lemma}

\subsection{Bounds using randomisation-improved Strichartz.}
In this section we will prove that under the conditions of Theorem \ref{random_thm}, we have
\begin{equation}
F^\om\in L^\infty_tL^{\f{2d}{d-2}}_x\medcap L^\f{1}{\sigma }_t L^{\f{2d}{d-4\sigma }}_x\medcap R(\R)\label{BOUNDS_ONE}
\end{equation}
and
\begin{equation}
\na F^\om\in L^2_t L^\f{4d-2}{2d-3-\sigma }_x\medcap L^2_tL^\f{2d(2d-1)}{2d^2-7d+4+d\sigma }_x\medcap L^2_tL^\f{2d(d-2)}{d(d-6)+16\sigma }_x(\R)\label{BOUNDS_TWO}
\end{equation}
almost surely. Note that it is not claimed in the Theorem that $\na F^\omega$ lies in the final space listed above, but we need it in order to deduce one of the other bounds by interpolation.

These results rely on the following proposition allowing us to gain derivatives on the randomised free evolution, adapted from \cite{spitz2021almost}. Throughout this section, $\fomega$ always refers to the randomisation of $f$ as in \eqref{rand_f} and all spacetime norms are over $\R \times \R^d$.

The key estimate for this section then reads as follows.
\begin{proposition}[See Proposition 3.4(ii), \cite{spitz2021almost}]	\label{randomisationestimate}
Let $(q,p_0)\in [2,\infty)$ satisfy
\begin{align}\label{A}
\frac{1}{q}\leq\left(d-\frac{1}{2}\right)\left(\frac{1}{2}-\frac{1}{{p_0}}\right) && \text{ and } && ({q},{p_0})\neq\left(2,\frac{4d-2}{2d-3}\right)
\end{align}
Let $p\in[p_0, \infty)$. Then for any $f\in H^{s}(\R^d)$ with $s\geq 0$, it holds
$$\|\Ut f^{\omega}\|_{L^{\beta}_{\omega}\dot{B}^{s+\frac{2}{q}+\frac{d}{p_0}-\frac{d}{2}}_{q,p,2}}\lesssim_{d,q,p,p_0} \sqrt{\beta} \|f\|_{H^s(\R^d)}$$
for all $\beta\in[1,\infty)$.
\end{proposition}

Observe that the maximum derivative gain by this estimate occurs at the non-allowed endpoint $(2,\f{4d-2}{2d-3})$, where we would gain $$\f{2}{q}+\f{d}{p_0}-\f{d}{2}=\f{d-1}{2d-1}$$ derivatives.

Since the proof of Proposition \ref{randomisationestimate} is very similar to the $d=4$ case in \cite{spitz2021almost}, we omit it here, however we remark that it relies crucially on a Strichartz estimate in radially averaged spaces due to Guo \cite{guo2016sharp}.

The bounds on $F^\om$ without any derivatives are then implied by the following corollary:
\begin{corollary}\label{cor741}
Let $q\in[2,\infty)$, $p\in[2,\infty)$ satisfy
\begin{align}       \label{nonstrict}
\frac{2}{q}+\frac{d}{p}\leq \frac{d}{2}
\end{align}
Let $f\in L^2(\R^d)$, $\fomega$ its randomisation. Then for almost every $\om\in\Omega$ it holds 
\[
\|\Ut \fomega\|_{L^q_tL^p_x}<\infty
\]
\end{corollary}
\begin{proof}
Let $\beta \geq 2$. By the Littlewood-Paley inequality we have
\begin{equation*}
    \|\Ut f^\om\|_{L^\beta_\om L^q_t L^p_x}\lesssim_d \left\| \left( \sum_N \|P_N \Ut f^\om\|_{L^q_t L^p_x}\right)^\frac{1}{2}\right\|_{L^\beta_\om}=\|\Ut f^\om\|_{L^\beta_\om \db^0_{q,p,2}}
\end{equation*}
Since $(q,p)$ satisfy \eqref{nonstrict}, there exists $2\leq p_0\leq p$ such that $(q,p_0)$ is a Strichartz pair, i.e. $$\frac{2}{q}+\frac{d}{p_0}-\frac{d}{2}=0$$
Since every Strichartz pair satisfies \eqref{A}, we are able to immediately apply Proposition \ref{randomisationestimate} with $s=0$ to obtain the result.
\end{proof}

It is then immediate that, for $\sigma (d)$ sufficiently small, $F^\om\in L^{\f{1}{\sigma }}_t L^{\f{2d}{d-4\sigma }}_x$ almost surely.

To show that $F^\omega\in L^\infty_t L^{\f{2d}{d-2}}_x(\R)$ for almost every $\om$ requires an endpoint case of Proposition \ref{randomisationestimate} allowing for $q=\infty$. We prove this as in \cite{killip2019almost}.
\begin{lemma}
Let $s\geq\f{d-2}{2d-1}$, $f\in H^s(\R^d)$, $\fomega$ its randomisation. Then for almost every $\omega\in\Omega$ it holds $$\Ut f^\omega\in L^\infty_t L^{\pp}_x(\R^d)$$
\end{lemma}
\begin{proof}
Let $I\subset \R$ with $|I|=\dl$ to be determined. Let $t_0$, $t\in I$. Then for any $N\in 2^{\Z}$ we have
$$\|P_N\Ut \fomega\|_{L^{\pp}_x(\R^d)}\leq\|P_N\Ut \fomega(t_0)\|_{L^{\pp}_x(\rd)}+\|\dd_tP_N\Ut \fomega\|_{L^1_t L^{\pp}_x(I)}$$
and averaging this over $t_0\in I$ we find
\begin{align*}
\|P_N \Ut \fomega\|_{L^{\pp}_x(\rd)}\lesssim& \dl^{-1}\|P_N \Ut \fomega\|_{1,\pp[I]}+\|\dd_tP_N \Ut \fomega\|_{1,\pp[I]}\\
\lesssim&\dl^{-\f{d-2}{2d}}\|P_N \Ut \fomega\|_{\pp,\pp[\R]}+\dl^{\f{d+2}{2d}}\|\dd_tP_N \Ut \fomega\|_{\pp,\pp[\R]}\\
\lesssim&\dl^{-\f{d-2}{2d}}\|P_N \Ut \fomega\|_{\pp,\pp[\R]}+\dl^{\f{d+2}{2d}}N^2\|P_N\Ut \fomega\|_{\pp,\pp[\R]}\\
\lesssim&N^\f{d-2}{d}\|P_N\Ut \fomega\|_{\pp,\pp[\R]}
\end{align*}
choosing $\dl=N^{-2}$ in the last line.

Averaging over $\om$ we thus obtain, via the Littlewood-Paley inequality,
\begin{align*}
\|\Ut \fomega\|_{L^\beta_\om L^\infty_t L^{\pp}_x(\R)}\lesssim&\left\|\left(\sum_N\|P_N\Ut \fomega\|_{\infty,\pp[\R]}^2\right)^\frac{1}{2}\right\|_{L^\beta_\om}\\
\lesssim&\left\|\left(\sum_N(N^{\f{d-2}{d}}\|P_N\Ut \fomega\|_{\pp,\pp})^2\right)^\frac{1}{2}\right\|_{L^\beta_\om}\\
=&\|\Ut \fomega\|_{L^\beta_\om \dot{B}^\f{d-2}{d}_{\pp,\pp,2}}
\end{align*}
Applying Proposition \ref{randomisationestimate} with $s=\f{d-2}{2d-1}$, $q=p=\pp$ and $p_0=\f{2d(2d-1)}{2d^2-3d+4}$ we see that this is bounded by $\sqrt{\beta}\|f\|_{H^s}$.
\end{proof}

For the remaining bounds of \eqref{BOUNDS_ONE}-\eqref{BOUNDS_TWO}, we prove another corollary of Proposition \ref{randomisationestimate} to handle the terms involving $\na F^\om$.
\begin{corollary}\label{cor1006}
Let $q,p\in [2,\infty)$ satisfy \eqref{A}, $s>1-\left(\f{d-1}{2d-1}\right)\f{2}{q}\in (0,1)$. Then for any $f\in H^s(\R^d)$, $\fomega$ its randomisation, we have
$$\|\na \Ut\fomega\|_{L^q_tL^p_x}<\infty$$ for almost every $\om\in\Omega$.
\end{corollary}
\begin{proof}
Setting $p_0=\left(\frac{1}{2}-\f{2}{2d-1}\cdot\f{1}{q}\right)^{-1}\in(2,\infty)$ for $q\neq 2$, and $p_0=\f{4d-2}{2d-3-\dl}$ for $q=2$, where $\dl$ is a small constant, the pair $(q,p_0)$ satisfies \eqref{A}
and
\begin{equation*}
\f{2}{q}+\f{d}{p_0}-\f{d}{2}=
\begin{cases}
\left(\f{d-1}{2d-1}\right)\f{2}{q}  \text{ for } q\neq 2\\
\f{d-1}{2d-1}-\f{\dl}{2}\f{d}{2d-1} \text{ for }q=2
\end{cases}  
\end{equation*}
Applying Proposition \ref{randomisationestimate} in combination with the Littlewood-Paley inequality for these parameters yields the result, provided $\dl(d,p,s)$ is sufficiently small.
\end{proof}

Applying this corollary with $q=2$ we immediately obtain that $$\na F^\om\in L^2_tL^{\f{4d-2}{2d-3-\sigma }}_x\medcap L^2_t L^{\f{2d(2d-1)}{2d^2-7d+4+d\sigma }}_x\medcap L^2_tL^{\f{2d(d-2)}{d(d-6)+16\sigma }}_x(\R)$$ completing the proof of \eqref{BOUNDS_ONE}-\eqref{BOUNDS_TWO} since $s>1-\left(\f{d-1}{2d-1}\right)$.

To conclude this section, we show that $\|F^\om\|_{R(\R)}<\infty$ almost everywhere.

\begin{proposition}
Let $s>\frac{d^2+6d-4}{(2d-1)(d+2)}$. Then for almost every $\om \in \Omega$ we have
$$F^\omega \in  R(\R)$$
\end{proposition}

\begin{proof}
Recall $$\|F^{\omega}\|_{R(\R)}:=\|F^{\omega}\|_{\f{2(d+2)}{d-2},\f{2d(d+2)}{d^2+4}(\R)}+\|\na F^{\omega}\|_{\f{2(d+2)}{d-2},\f{2d(d+2)}{d^2+4}(\R)}+\|F^{\omega}\|_{\db^{\frac{4}{d+2}}_{d+2,\frac{2(d+2)}{d}}(\R)}$$ For the first terms, we apply Corollaries \ref{cor741} and \ref{cor1006} with $q=\frac{2(d+2)}{d-2}$, $p=\f{2d(d+2)}{d^2+4}$, which forces the lower bound $$s>\frac{d^2+6d-4}{(2d-1)(d+2)}$$

For the final term, apply Proposition \ref{randomisationestimate} with $q=d+2$, $p=\f{2(d+2)}{d}$ and $p_0=\frac{2(2d-1)(d+2)}{2d^2+3d-6}\in[2,p]$ to obtain
$$\|F^{\omega}\|_{L^\beta_\om \db^{\frac{4}{d+2}}_{d+2,\frac{2(d+2)}{d},2}}\lesssim \sqrt{\beta}\|f\|_{H^s(\R^d)}$$ 
for any $s>\frac{2(3d-1)}{(d+2)(2d-1)}$ and $\beta \geq 1$.
\end{proof}

\subsection{Randomised $L^1_t$ estimate}\label{l1t}
In this section we will prove that, under the conditions of Theorem \ref{random_thm}, we have
\begin{equation}\label{BOUNDS_FOUR}
 \na F^{\om}\in L^1_t L^{\f{2d}{d-4}}_x \medcap L^{\f{d-2}{d-2-4\sigma }}_tL^{\f{2d(d-2)}{d(d-6)+16\sigma }}_x(\R)
\end{equation}
almost surely. Since we have already proved that $\na F^\om\in L^2_tL^{\f{2d(d-2)}{d(d-6)+16\sigma }}_x(\R)$ almost surely, for the third bound it is sufficient to prove that
\begin{equation}\label{BOUND_FIVE}
 \na F^{\om}\in L^1_tL^{\f{2d(d-2)}{d(d-6)+16\sigma }}_x(\R)
\end{equation}
We thus only need to find estimates in Lebesgue spaces with time exponent $1$. Key to such bounds are the following propositions which are generalisations of results of Spitz \cite{spitz2021almost} to high dimensions. The proofs are the same as in the dimension 4 case so we do not present them here, however we remark that it is for these results that the physical space part of the randomisation of $f$ is necessary.

The first result exploits the decay properties of the Schr\"{o}dinger semi-group to achieve bounds in spaces with low time integrability away from $t=0$:
\begin{proposition}[Proposition 3.6, \cite{spitz2021almost}] \label{prop36}
Let $s\geq 0$ and consider $q\in [1,\infty)$, $p\in [2,\infty)$ $\sigma \geq 0$ such that $$\sigma <\f{d}{2}-\f{1}{q}-\f{d}{p}$$
Let $f\in H^s(\R^d)$ and $\fomega$ be its randomisation as in \eqref{rand_f}. Then it holds
$$\|t^\sigma  \Ut \fomega\|_{L^\beta_\om L^q_{t}\db^s_{p,2}([1,\infty))}\lesssim_{d,q,p,\sigma } \sqrt{\beta} \|f\|_{H^s(\R^d)}$$
for all $\beta \geq 1$.
\end{proposition}

The gain in derivatives needed for \eqref{BOUNDS_FOUR} is obtained by interpolating this with the improved Strichartz estimate of Proposition \ref{randomisationestimate} to obtain the following:
\begin{proposition}[Proposition 3.7, \cite{spitz2021almost}]\label{prop37}
Let $s>\f{d+1}{2d-1}$. Then for any $f\in H^s(\R^d)$, $\fomega$ its randomisation, it holds $$\|\na \Ut\fomega\|_{L^\beta_\om L^1_t L^\infty_x(\R)}\lesssim_{s,d} \sqrt{\beta}\|f\|_{H^s(\R^d) }$$
for all $\beta\geq1$.
\end{proposition}

To prove \eqref{BOUNDS_FOUR}, we need a more general version of this proposition allowing for a larger range of exponents in the $x$-variable. The proof is a modification of the proof in \cite{spitz2021almost} of the previous result.
\begin{proposition}
Let $s>\f{4d-1}{3(2d-1)}$, $\beta\geq 1$. Then for any $f\in H^s(\R^d)$, $\fomega$ its randomisation, it holds
\begin{equation}\label{eqn60}
\|\na \Ut \fomega\|_{L^\beta_\om L^1_t L^r_x(\R)}\lesssim_{d,s,r}\sqrt{\beta} \|f\|_{H^s(\R^d) }
\end{equation}
for any $\f{2d}{d-4}\leq r\leq\infty$.
\end{proposition}
\begin{proof}
We will prove the case $r=\f{2d}{d-4}$. This is sufficient by interpolation with Proposition \ref{prop37}. 

Observe that we may decompose the left hand side of \eqref{eqn60} as 
\begin{align*}
    &\|\na \Ut \fomega\|_{L^\beta_\om L^1_t L^{\f{2d}{d-4}}_x(\R)}\lesssim \|\Ut \fomega\|_{L^\beta_\om L^1_t \db^1_{\f{2d}{d-4},2}(\R)}\\
    \lesssim&\|\Ut \fomega\|_{L^\beta_\om L^1_{t}\db^1_{\f{2d}{d-4},2}(-1,1)}+\|\Ut \fomega\|_{L^\beta_\om L^1_{t}\db^1_{\f{2d}{d-4},2}(-\infty,-1]}+\|\Ut \fomega\|_{L^\beta_\om L^1_{t}\db^1_{\f{2d}{d-4},2}[1,\infty)}
\end{align*}
We first consider the term over $(-1,1)$. By H\"{o}lder's inequality we have
\begin{equation*}
    \|\Ut \fomega\|_{L^{\beta}_{\om} L^1_{t} \db^1_{\f{2d}{d-4},2}(-1,1)}\lesssim_d\|\Ut\fomega\|_{L^\beta_\om L^2_{t} \db^1_{\f{2d}{d-4},2}(-1,1)}    \lesssim_d \|\Ut\fomega\|_{L^\beta_\omega \db^{\nu(\dl)+\beta(\dl)}_{2,\f{2d}{d-4},2}(-1,1)}
\end{equation*}
where 
\begin{align*}
\beta(\dl)=\f{d-1}{2d-1}-\f{\dl}{2}\f{d}{2d-1} &&\text{ and }&&\nu(\dl)=\f{d}{2d-1}+\f{\dl}{2}\f{d}{2d-1}    
\end{align*}
for some $0<\dl(d,s)\ll 1$ to be determined.

We may now apply Proposition \ref{randomisationestimate} with $q=2$, $p_0=\f{4d-2}{2d-3-\dl}$, $p=\f{2d}{d-4}$ to obtain
\begin{equation*}
    \|\Ut \fomega\|_{L^{\beta}_{\om} L^1_{t} \db^1_{\f{2d}{d-4},2}(-1,1)}\lesssim_d\sqrt{\beta}\|f\|_{H^{\nu(\dl)}}
    \lesssim_d \sqrt{\beta} \|f\|_{H^s}
\end{equation*}
for $\dl(s,d)$ sufficiently small.

Next consider the term over $[1,\infty)$. By H\"{o}lder's inequality for sequences we have
\begin{equation*}
    \|\Ut \fomega\|_{\db^1_{\f{2d}{d-4},2}}=\left(\sum_N N^2\|P_N\Ut\fomega\|_{\f{2d}{d-4}}^2\right)^\frac{1}{2}
    \leq \|\Ut \fomega\|_{\db^{1+\gamma}_{\f{2d}{d-4},2}}^\al\|\Ut\fomega\|_{\db^{1-\f{\al\gamma}{1-\al}}_{\f{2d}{d-4},2}}^{1-\al}
\end{equation*}
for $\al\in [0,1)$, $\gamma\in [0,\f{1-\al}{\al}$) to be determined.

Combining this with H\"{o}lder's inequality in time, we have
\begin{align}
    \|\Ut\fomega\|_{L^\beta_\om L^1_t \db^1_{\f{2d}{d-4},2}[1,\infty)}\lesssim&_\dl \left\|t^{\f{1+\dl}{2}}\|\Ut\fomega\|_{\db^{1+\gamma}_{\f{2d}{d-4},2}}^\al\|\Ut\fomega\|_{\db^{1-\f{\al \gamma}{1-\al}}_{\f{2d}{d-4},2}}^{1-\al}\right\|_{L^\beta_\om L^2_t[1,\infty)}\nonumber\\
    \lesssim&_\dl \|\Ut\fomega\|_{L^\beta_\om L^2_t \db^{1+\gamma}_{\f{2d}{d-4},2}[1,\infty)}^\al \|t^\f{1+\dl}{2(1-\al)} \Ut\fomega\|_{L^\beta_\om L^2_t \db^{1-\f{\al\gamma}{1-\al}}_{\f{2d}{d-4},2}[1,\infty)}^{1-\al} \label{eqn75}
\end{align}

We will bound the first term of \eqref{eqn75} using the randomisation-improved Strichartz estimate from Proposition \ref{randomisationestimate}, and the second term using Proposition \ref{prop36}. Fix 
\begin{align*}
\al=\f{2}{3}-\dl&&\text{ and }&&\gamma=\f{d-1}{3(2d-1)}    
\end{align*}
(chosen to optimise the gain in derivatives in what follows). 
Applying Proposition \ref{randomisationestimate} with $q=2$, $p_0=\f{4d-2}{2d-3-\dl}$ and $p=\f{2d}{d-4}$ we obtain, for $\beta(\dl)$ and $\nu(\dl)$ as before,
\begin{align*}
\|\Ut\fomega\|_{L^\beta_\om L^2_t \db^{1+\gamma}_{\f{2d}{d-4},2}[1,\infty)}&\lesssim_d\|\Ut\fomega\|_{L^\beta_\om \db^{\gamma+\nu(\dl)+\beta(\dl)}_{2,\f{2d}{d-4},2}[1,\infty)}\\
&\lesssim_{d,\dl} \sqrt{\beta}\|f\|_{H^{\gamma+\nu(\dl)}}\\
&\lesssim_{d,\dl} \sqrt{\beta}\|f\|_{H^s}
\end{align*}
since $\gamma+\nu(\dl)=\f{4d-1}{3(2d-1)}+\f{\dl}{2}\f{d}{2d-1}<s$ for $\dl$ sufficiently small.

For the second term we apply Proposition \ref{prop36} with $q=2$, $p=\f{2d}{d-4}$, $\sigma =\f{1+\dl}{2(1-\al)}=\f{3+3\dl}{2+6\dl}$ to find
\begin{align*}
    \|t^\f{1+\dl}{2(1-\al)}\Ut\fomega\|_{L^\beta_\om L^2_t \db^{1-\f{\al \gamma}{1-\al}}_{\f{2d}{d-4},2}[1,\infty)}\lesssim& \sqrt{\beta}\|f\|_{H^s}
\end{align*}
since $1-\f{\al \gamma}{1-\al}=\f{4d-1}{3(2d-1)}+O(\dl)<s$ for $\dl$ sufficiently small.
Returning to \eqref{eqn75} we have $$\|\Ut\fomega\|_{L^\beta_\om L^1_t \db^1_{\f{2d}{d-4},2}[1,\infty)}\lesssim_d \sqrt{\beta}\|f\|_{H^s}$$

Treating the term over $(-\infty,-1]$ in the same way we obtain the desired result.
\end{proof}

The bounds \eqref{BOUNDS_FOUR} (via \eqref{BOUND_FIVE}) are now immediate, observing that $r=\f{2d(d-2)}{d(d-6)+16\sigma }$ is greater than $\f{2d}{d-4}$ for $\sigma (d)$ sufficiently small.

\appendix
\section{Calculation of energy increment \eqref{oneone}-\eqref{twotwo}}\label{appendix0}

\begin{proposition}\label{prop1}
Let $\vn\in C^1(\R,L^2)\cap C(\R,H^2)$ solve \eqref{FNLSn3} for some $F_n$ satisfying the conditions of Theorem \ref{thm334}. Then for any $T_1$, $T_2\in\R$ it holds 
\begin{multline*}
E_n(v_n(T_2))-E_n(v_n(T_1))
=-\frac{1}{2}\left[\intr\vp_n(|F_n+v_n|^2)-\vp_n(|v_n|^2)-\vp_n(|F_n|^2)dx\right]_{T_1}^{T_2}\\
-\im\int_{T_1}^{T_2}\intr\na \overline{F_n}\cdot\na(g_n(F_n+v_n)-g_n(F_n))dxdt
\end{multline*}
\end{proposition}
Before proving this proposition, we recall without proof the following useful fact:

Let $X$ be a Banach space. Then any $f\in C^1(\R,X)$ is in fact Fr\'{e}chet differentiable from $\R$ to $X$ with Fr\'{e}chet derivative $\dd_tf(t,\cdot)$ (see, for example, Section 1.3 \cite{cazenave2003semilinear}).\footnote{$C^1(\R,X)$ is defined anaologously to \ref{C1_space}.}

We will also use the following result to differentiate the nonlinearity:
\begin{lemma}\label{lem3}
Let $\psi\in C^2(\C,\C)$ with bounded second derivatives. Suppose also $\psi(w)\in L^1(\R^d)$ for all $w\in \lt(\R^d)$. Then the map $$H:w\mapsto \int_{\R^d}\psi(w(x))dx$$ is Fr\'{e}chet differentiable from $\lt(\R^d)$ to $\R$ with derivative
$$
DH|_w(h)=\int_{\R^d}(h\dd_z\psi(w)+\bar{h}\dd_{\bar{z}}\psi(w))dx
$$
\end{lemma}

Applying this lemma with $\psi(z)=\vp_n(|z|^2)$ and using the chain rule we observe that for any $v\in C^1(\R,L^2(\rd))$ it holds
\begin{equation}\label{differentiating_g_integral}
\dd_t\intr\vp_n(|v|^2)dx=2\re\intr\dd_t\overline{v} \text{ }g_n(v)dx
\end{equation}

We can now prove Proposition \ref{prop1}.
\begin{proof}[Proof of Proposition \ref{prop1}] 
We split the energy into a kinetic and a potential term:
\begin{align*}
KE(v):=\frac{1}{2}\langle \na v, \na v\rangle_{L^2}&&\text{ and }&&G_n(u):=\frac{1}{2}\intr\vp_n(|u|^2) dx
\end{align*}
The map $KE$ is (Fr\'{e}chet) differentiable from $H^1(\R^d)$ to $\R$, thus if we further suppose that $v\in C^1(\R,H^1)$ we have
\begin{equation*}\label{KEder}
\f{d}{dt}KE(v(t))=\re\intr\overline{\na v}\na(\dd_t v)dx=-\re\intr\dd_t\bar{v}\Delta v dx
\end{equation*}
and the same formula holds for $\vn\in C^1(\R,L^2(\R^d))\cap C(\R,H^2(\R^d))$ by approximation.

Likewise, by \eqref{differentiating_g_integral}, we have
\begin{equation*}\label{PEder}
\f{d}{dt}G_n(\vn(t))=\re\intr\dd_t\overline{\vn}g_n(v_n)dx
\end{equation*}

Combining these results and using that $\vn$ satisfies equation \eqref{FNLSn3} we obtain
\begin{align*}
E_n(\vn(T_2))&-E_n(\vn(T_1))=\re\int_{T_1}^{T_2}\intr\dd_t\overline{\vn}[-\Delta\vn+g_n(v_n)]dxdt\\
=&-\re\int_{T_1}^{T_2}\intr\dd_t\overline{\vn}[g_n(\fvn)-g_n(\vn)]dxdt\\
=&-\re\int_{T_1}^{T_2}\intr (\dd_t(\overline{F_n+v_n})g_n(F_n+v_n)-\dd_t\overline{v_n}g_n(v_n)-\dd_t\overline{F_n}g_n(F_n)) dxdt\\
&+\re\int_{T_1}^{T_2}\intr\dd_t\overline{F_n}(g_n(F_n+v_n)-g_n(F_n))dxdt
\label{equation_nought}
\end{align*}
It then remains to use \eqref{differentiating_g_integral} to rewrite the first integral above, and $\dd_tF_n=i\Delta F_n$ for the second, to obtain
\begin{align*}
&E_n(\vn(T_2))-E_n(\vn(T_1))\\
&=-\frac{1}{2}\int_{T_1}^{T_2}\dd_t\intr\vp_n(|F_n+v_n|^2)-\vp_n(|v_n|^2)-\vp_n(|F_n|^2)dxdt\\
&\quad+\im\int_{T_1}^{T_2}\intr\Delta \overline{F_n}(g_n(F_n+v_n)-g_n(F_n))dxdt\\
&=-\frac{1}{2}\left[\intr\vp_n(|F_n+v_n|^2)-\vp_n(|v_n|^2)-\vp_n(|F_n|^2)dx\right]_{T_1}^{T_2}\\
&\quad-\im\int_{T_1}^{T_2}\intr\na \overline{F_n}\cdot\na(g_n(F_n+v_n)-g_n(F_n))dxdt
\end{align*}
where we used the fundamental theorem of calculus and integrated by parts to obtain the final equality.
\end{proof}

\section{Justification of Remark \ref{rmk1.5}}\label{appendixB}
Here we will outline the proof of the following statement claimed in Remark \ref{rmk1.5}.
\begin{lemma}
Let $0<s<1$. Let $f\in L^2(\rd)$ and $\fomega$ denote its randomisation \eqref{rand_f}. Then if the probability space $(\Om,\mathcal{A},\mathbb{P})$ and the random variables $X^M_{i,j,k,l}:\Om\rightarrow \R$ are as described in Remark \ref{rmk1.5}, we have that $f\notin H^s(\rd)$ implies $\fomega\notin H^s(\rd)$ for almost every $\om\in\Om$.
\end{lemma}

We will follow the method of \cite{burq2008random} (Appendix B), using the almost-orthogonality of the projections $P_j$ to estimate
\begin{align*}
&\int_{\Om_1\times\Om_2\times\Om_3}e^{-\|\fomega\|_{H^s}^2}d\mathbb{P}(\om_1,\om_2,\om_3)\\
\leq&\int_{\Om_2\times\Om_3}\int_{\Om_1}e^{-{C}\sum_j|X_j(\om_1)|^2\|P_j\sum_{k,l,M,i}X^M_{k,l}(\om_2)X_i(\om_3)f^{M,i}_{k,l}\|_{H^s}^2}d\mathbb{P}(\om_1)d\mathbb{P}(\om_2,\om_3)\\
\leq&\int_{\Om_2\times\Om_3}e^{-C\|\sum_{k,l,M}X^M_{k,l}(\om_2)\sum_i X_i(\om_3)f^{M,i}_{k,l}\|_{H^s}^2}d\mathbb{P}(\om_2,\om_3)
\end{align*}
where we used that the random variables $X_j$ take values in $\{\pm1\}$.

We can treat the integral over $\Om_2$ similarly to find
$$
\int_{\Om_1\times\Om_2\times\Om_3}e^{-\|\fomega\|_{H^s}^2}d\mathbb{P}(\om_1,\om_2,\om_3)\leq\int_{\Om_3}e^{-C\|\sum_i X_i(\om_3)\vpi f\|_{H^s}^2}d\mathbb{P}(\om_3)
$$

In order to repeat the argument on $\Om_3$ and complete the proof, it remains to prove
\begin{equation}\label{double_inequality}
\|f\|_{H^s}^2\lesssim_{s,d}\sum_i\|\vpi f\|_{H^s}^2\lesssim_{s,d}\|\sum_i X_i(\om_3)\vpi f\|_{H^s}^2
\end{equation}

We will only prove the second inequality, the first being similar.
Denote $f_i:=\xw f$ and write
\begin{align*}
\|\vpi f\|_{\ell^2_iH^s_x(\rd)}=\|\vpi f_i\|_{\ell^2_iH^s_x(\rd)}\leq \|P_{\leq M_0}\vpi f_i\|_{\ell^2_iH^s_x(\rd)}+\|P_{>\mo}\vpi f_i\|_{\ell^2_iH^s_x(\rd)}
\end{align*}
for some $\mo\geq1$.
We first handle the low frequency contributions using the almost-orthogonality of the $(\vpi)_i$:
\begin{align*}
\|P_{\leq M_0}\vpi f_i\|_{\ell^2_iH^s_x}\lesssim \langle \mo\rangle^s\|\vpi f_i\|_{\ell^2_iL^2_x}\lesssim \langle \mo\rangle^s\|\sum_i\vpi f_i\|_{L^2}
\end{align*}

For the high frequency contributions we have to introduce the commutators $[P_M,\vpi]$. We have
\begin{align*}
\|P_{>\mo}\vpi f_i\|_{\ell^2_i H^s_x}\lesssim&\| M^s P_M \vpi f_i\|_{\ell^2_{M>\mo} \ell^2_i  L^2_x}\\
\lesssim&\|  M^s \vpi P_M f_i\|_{\ell^2_{M>\mo} \ell^2_i  L^2_x}+\underbrace{\|  M^s [P_M,\vpi] f_i\|_{\ell^2_{M>\mo} \ell^2_i  L^2_x}}_{(A)}\\
\lesssim&\|  M^s \sum_i \vpi P_M f_i\|_{\ell^2_{M>\mo}  L^2_x}+(A)\\
\lesssim&\|P_{>\mo} \sum_i \vpi f_i\|_{H^s_x}+\underbrace{\| M^s \sum_i [\vpi,P_M] f_i\|_{\ell^2_{M>\mo}  L^2_x}}_{(B)}+(A)\\
\lesssim&\|\sum_i \vpi f_i\|_{H^s_x}+(A)+(B)
\end{align*}
so we must bound (A) and (B).

\sloppy Let's study (A) first. Since the $(X_i)_i$ take values in $\{\pm1\}$ we have that $|[P_M,\vpi]f_i|=|[P_M,\vpi]f|$ so
\begin{multline}\label{huh}
(A)=\| M^s [P_M,\vpi] f\|_{\ell^2_{M>\mo} \ell^2_i  L^2_x}\\
\lesssim\| M^s [P_M,\vpi]\tvpi f\|_{\ell^2_{M>\mo} \ell^2_i  L^2_x}+\| M^s \vpi P_M \sum_{j:|j-i|> 8\sqd}\vpj f\|_{\ell^2_{M>\mo} \ell^2_i  L^2_x}
\end{multline}
for $\tvpi=\sum_{j:|j-i|\leq8\sqd}\vpj$, so that $\tvpi\vpi=1$.
To handle the first term, we will use the bound
\begin{align}\label{first_A_bound}
\|[P_M,\vpi]g\|_2\lesssim_d M^{-1}\|\na\vpi\|_{\infty}\|g\|_2\lesssim_d M^{-1}\|g\|_2
\end{align}
which follows from writing the frequency projection as a convolution operator.
For the second term we will use a slightly stronger form of Lemma 3.2 from \cite{spitz2021almost}, holding for any $D>0$, $|i-j|\geq 8\sqd$:
\begin{equation}\label{second_A_bound}
\|\vpi P_M\vpj g\|_2\lesssim_{d,D} M^{-D}|i-j|^{-D}\|\vpj g\|_2
\end{equation}

With these results in hand, we are able to bound (A). 
Using \eqref{first_A_bound} on the first term and the triangle inequality followed by \eqref{second_A_bound} on the second term we obtain
\begin{align*}
(A)&\lesssim\| M^{s-1} \tvpi f\|_{\ell^2_{M>\mo} \ell^2_i  L^2_x}+\| M^{s-D}|i-j|^{-D} \vpj f\|_{\ell^2_{M>\mo} \ell^2_i \ell^1_{j:|j-i|>8\sqd} L^2_x}\\
\lesssim& \mo^{s-1}\|f\|_2+\mo^{s-D}\left\| \||i-j|^{-D/2}\|_{\ell^2_{j:|j-i|>8\sqd}}\cdot\||i-j|^{-D/2} \vpj f\|_{\ell^2_{j:|j-i|>8\sqd}L^2_x}\right\|_{\ell^2_i}
\end{align*}
by the Cauchy-Schwarz inequality. For $D$ sufficiently large, $\||i-j|^{-D/2}\|_{\ell^2_{j:|j-i|>8\sqd}}$ is finite and we may swap the sums over $i$ and $j$ in the second term to obtain
\begin{align*}
(A)\lesssim&\mo^{s-1}\|f\|_2+\mo^{s-D}\||i-j|^{-D/2} \vpj f\|_{\ell^2_j \ell^2_{i:|i-j|>8\sqd}L^2_x}\\
\lesssim&\mo^{s-1}\|f\|_{L^2(\rd)}
\end{align*}

We now turn to
$$
(B)=\|M^s\sum_i[P_M,\vpi]f_i\|_{\ell^2_{M>\mo}L^2_x}
$$

This time write
$$
[P_M,\vpi]=\tvpi[P_M,\vpi]+(1-\tvpi)P_M\vpi
$$
We have
\begin{align*}
\|M^s\sum_i\tvpi[P_M,\vpi]f_i\|_{\ell^2_{M>\mo}L^2_x}
\leq&\|M^s\tvpi[P_M,\vpi]f_i\|_{\ell^2_{M>\mo} \ell^2_i L^2_x}\leq(A)\lesssim\mo^{s-1}\|f\|_2
\end{align*}
since the $\tvpi$ are bounded. On the other hand, using that $f_i=\xw f$, we have
\begin{align*}
\|M^s\sum_i(1-\tvpi)P_M\vpi f_i\|_{\ell^2_{M>\mo} L^2_x}
=&\|M^s\sum_i\sum_{j:|j-i|>8\sqd}\vpj P_M\vpi f\|_{\ell^2_{M>\mo} L^2_x}\\
=&\|\sum_j\vpj \cdot M^sP_M(\sum_{i:|i-j|>8\sqd}\vpi) f\|_{\ell^2_{M>\mo} L^2_x}\\
\lesssim&\|\vpj \cdot M^sP_M(\sum_{i:|i-j|>8\sqd}\vpi) f\|_{\ell^2_{M>\mo} \ell^2_j L^2_x}
\end{align*}
by the almost-orthogonality of the projections $\vpj$. Using the triangle inequality followed by the estimate \eqref{second_A_bound}, we can bound the previous line by
\begin{align*}
&\|M^{s-D}|i-j|^{-D} \vpi f\|_{\ell^2_{M>\mo}\ell^2_j \ell^1_{i:|i-j|>8\sqd}L^2_x}\\
\lesssim&\mo^{s-D}\||i-j|^{-D} \vpi f\|_{\ell^2_j \ell^1_{i:|i-j|>8\sqd}L^2_x}\\
\lesssim&\mo^{s-D}\|\||i-j|^{-D/2}\|_{\ell^2_{i:|i-j|>8\sqd}} \cdot\||i-j|^{-D/2} \vpi f\|_{\ell^2_{i:|i-j|>8\sqd}L^2_x} \|_{\ell^2_j}\\
\lesssim&\mo^{s-D}\||i-j|^{-D/2} \vpi f\|_{\ell^2_i \ell^2_{j:|j-i|>8\sqd}L^2_x}\lesssim\mo^{s-1}\|f\|_2
\end{align*}
since $D,\mo\geq1$. This completes the estimate for (B).

Combining the estimates we have just found with the bound on the low frequency contributions we conclude that
\begin{align*}
\|\vpi f\|_{\ell^2_i H^s_x(\rd)}\lesssim&\langle\mo\rangle^s\|\sum_i\vpi f_i\|_{L^2(\rd)}+\|\sum_i\vpi f_i\|_{H^s(\rd)}+\mo^{s-1}\|f\|_{L^2(\rd)}\\
\lesssim&(1+\langle\mo\rangle^s)\|\sum_i \xw\vpi f\|_{H^s(\rd)}+\mo^{s-1}\|\vpi f\|_{\ell^2_i L^2_x(\rd)}
\end{align*}
Now, since $\|\vpi f\|_{\ell^2_i L^2_x(\rd)}\lesssim\|\vpi f\|_{\ell^2_i H^s_x(\rd)}$ and $s<1$, if we take $\mo$ sufficiently large (depending only on $s$ and $d$) we may move this term to the left hand side to obtain
$$
\left(\sum_i\|\vpi f\|_{H^s(\rd)}^2\right)^\frac{1}{2}\lesssim_{s,d}\|\sum_i \xw\vpi f\|_{H^s(\rd)}
$$
which completes the proof of the second inequality in \eqref{double_inequality}.

  \bigskip
  \footnotesize
\textsc{B\^{A}timent des Math\'{e}matiques, EPFL, Station 8, 1015 Lausanne, Switzerland}\par\nopagebreak
\textit{E-mail address}: \texttt{katie.marsden@epfl.ch}\par\nopagebreak
This article has been published in a revised form in Communications on Pure and Applied Analysis [\url{http://dx.doi.org/10.3934/cpaa.2023106}]. This version is free to download for private research and study only. Not for redistribution, re-sale or use in derivative works.

\end{document}